\documentclass[10pt]{amsart}

\usepackage[sc]{mathpazo}

\usepackage{mathtools}
\usepackage{amssymb}
\usepackage{amsfonts}

\usepackage{graphicx}
\usepackage{hyperref}
\usepackage{tabularx}
\usepackage{multirow}
\usepackage{tikz-cd}
\usepackage{svg}
\usepackage{xcolor}
\usepackage{enumitem}
\usepackage{bm}
\usepackage{todonotes}

\DeclareFontFamily{U}{mathx}{}
\DeclareFontShape{U}{mathx}{m}{n}{<-> mathx10}{}
\DeclareSymbolFont{mathx}{U}{mathx}{m}{n}
\DeclareMathAccent{\widehat}{0}{mathx}{"70}
\DeclareMathAccent{\widecheck}{0}{mathx}{"71}

\theoremstyle{plain}
\newtheorem{theorem}{Theorem}
\newtheorem{proposition}{Proposition}[section]
\newtheorem{lemma}[proposition]{Lemma}
\newtheorem{corollary}[proposition]{Corollary}
\theoremstyle{definition}

\newtheorem{definition}[proposition]{Definition}

\newtheorem{example}[proposition]{Example}
\theoremstyle{remark}
\newtheorem{remark}[proposition]{Remark}

\newcommand{\Z}{\mathbb{Z}}
\newcommand{\Q}{\mathbb{Q}}
\newcommand{\R}{\mathbb{R}}

\newcommand{\RP}{\mathbb{RP}}
\newcommand{\CP}{\mathbb{CP}}

\newcommand{\fA}{\mathfrak{A}}
\newcommand{\fB}{\mathfrak{B}}
\newcommand{\fC}{\mathfrak{C}}

\newcommand{\fF}{\mathfrak{F}}

\newcommand{\fU}{\mathfrak{U}}

\newcommand{\fa}{\mathfrak{a}}
\newcommand{\fb}{\mathfrak{b}}
\newcommand{\fc}{\mathfrak{c}}

\newcommand{\fp}{\mathfrak{p}}
\newcommand{\fq}{\mathfrak{q}}

\newcommand{\fs}{\mathfrak{s}}

\newcommand{\bfS}{\mathbf{S}}

\newcommand{\bfd}{\mathbf{d}}

\newcommand{\cB}{\mathcal{B}}
\newcommand{\cC}{\mathcal{C}}
\newcommand{\cD}{\mathcal{D}}

\newcommand{\cF}{\mathcal{F}}
\newcommand{\cG}{\mathcal{G}}

\newcommand{\cK}{\mathcal{K}}
\newcommand{\cL}{\mathcal{L}}
\newcommand{\cM}{\mathcal{M}}
\newcommand{\cN}{\mathcal{N}}
\newcommand{\cO}{\mathcal{O}}

\newcommand{\cS}{\mathcal{S}}
\newcommand{\cT}{\mathcal{T}}

\newcommand{\cV}{\mathcal{V}}

\newcommand{\bJ}{\mathbb{J}}

\newcommand{\colim}{\operatorname{colim}}
\newcommand{\im}{\operatorname{im}}
\newcommand{\coker}{\operatorname{coker}}

\DeclareMathOperator{\Tr}{Tr}    

\newcommand{\fo}{\mathfrak{o}}

\newcommand{\Spin}{\operatorname{Spin}}
\newcommand{\Pin}{\operatorname{Pin}}

\newcommand{\spin}{\mathrm{spin}}
\newcommand{\spinc}{\spin^{c}}

\newcommand{\CMbar}{\overline{CM}}
\newcommand{\CMto}{\widecheck{CM}}
\newcommand{\CMfrom}{\widehat{CM}}
\newcommand{\HMbar}{\overline{HM}}
\newcommand{\HMto}{\widecheck{HM}}
\newcommand{\HMfrom}{\widehat{HM}}

\newcommand{\HMred}{HM}

\newcommand{\dbar}{\bar{\partial}}
\newcommand{\dcheck}{\widecheck{\partial}}
\newcommand{\dhat}{\widehat{\partial}}

\usepackage{nicematrix}

\newcommand{\tildeX}{\tilde{X}}

\newcommand{\ucM}{\breve{\cM}}
\newcommand{\vinfty}{\bm{\infty}}
\newcommand{\W}{\breve{W}}
\newcommand{\EW}{E\W}
\newcommand{\ev}{\operatorname{ev}}

\newcommand{\Cr}{\mathfrak{Crit}}

\newcommand{\gr}{\operatorname{gr}}
\newcommand{\Gr}{\operatorname{Gr}}

\newcommand{\mbar}{\bar{m}}
\newcommand{\mcheck}{\widecheck{m}}
\newcommand{\mhat}{\widehat{m}}

\newcommand{\Wbar}{\overline{W}}
\newcommand{\Wcheck}{\widecheck{W}}

\title{Miyazawa's Invariant, Lefschetz Numbers, and Seifert Solids}
\author{Judson Kuhrman}

\begin{document}

\begin{abstract}
    We establish a formula expressing Miyazawa's 2-knot invariant $|\deg|$ in terms of the Lefschetz number of a map on ordinary (i.e., not real) monopole Floer homology. As an application, we deduce that $|\deg|=1$ for any 2-knot in $S^4$ which has a punctured $L$-space as a Seifert solid. In the course of the proof of the main theorem, we show how Francesco Lin's construction of monopole Floer homology with $\Pin(2)$-equivariant perturbations can be made to work with integer coefficients.
\end{abstract}

\maketitle

\vspace{-11pt}

\section{Introduction}\label{section:introduction}

The study of real Seiberg--Witten gauge theory has garnered much interest recently due to the apparent potential for applications to the study of exoticism in smooth 4-dimensional topology. Of particular interest is the $\Z^{\geq 0}$-valued real Bauer-Furuta degree invariant $|\deg|$, which is defined for certain 4-manifolds equipped with real $\spinc$ structures, and specializes to an invariant of 2-knots and $\RP^2$-knots embedded in $S^4$ by way of their double branched covers with their unique real $\spin$ structures. The invariant $|\deg|$, derived from the real Bauer-Furuta invariant defined in \cite{konno-miyazawa-taniguchi}, was used in \cite{miyazawa} to prove the existence of smoothly exotic, topologically unknotted embeddings $\RP^2\hookrightarrow S^4$. Despite the recent attention this invariant has received, methods for computing the invariant are sparse and it remains largely unclear what type of information it contains.

Let $\cS:S^2\hookrightarrow S^4$ be a (smooth) 2-knot. Suppose that $Y^\circ\hookrightarrow S^4$ is a Seifert solid for $\cS$, i.e., an orientable hypersurface with $\partial Y^\circ = \cS$, and let $Y$ be the corresponding closed 3-manifold. We can associate to $\cS,Y$ a cobordism $W:Y\to Y$, as follows. Let $X=X(\cS)$ be the 4-manifold obtained from surgery on $S^4$ along $\cS$. Then, we take $W$ to be $X$-cut-along-$Y$. The homology of $X$ is that of $S^1\times S^3$, and so there exists a unique double cover $\tildeX\to X$. Let $\fs_{X}$ be the unique $\spinc$ structure on $X$ and let $\fs_{Y},\fs_{W}$ be the $\spinc$ structures restricted to $Y,W$, respectively. The double cover $\tildeX$ can be built explicitly by gluing two copies of $W$ end to end. There is a unique real $\spinc$ structure on $\tildeX$, and the underlying $\spinc$ structure is that pulled back from $X$. The invariant $|\deg|$ is defined, and $|\deg(\tildeX)|=|\deg(\cS)|$ \cite[Proposition~7.3]{miyazawa-park-taniguchi}. In the case that $Y$ is a rational homology sphere, we will express the invariant $|\deg(\cS)|$ in terms of a cobordism map induced by $W$ on the ``ordinary" monopole Floer homology of $Y$, twisted by an involution $\jmath_*$. The $\spinc$ structure $\fs_Y$ is induced by a $\spin$ structure, and the involution $\jmath_*$ is the complex conjugation isomorphism induced by $\bar\fs_Y\cong\fs_Y$. Our main result is the following.

\begin{theorem}\label{thm:deg-formula}
    Let $X$ be a homology $S^1\times S^3$ and $\tildeX\to X$ its unique double cover. Suppose that there exists an embedded oriented rational homology 3-sphere $Y\hookrightarrow X$ representing a generator of $H_3(X;\Z)\cong\Z$. Let $W$ be the cobordism $X$-cut-along-$Y$ and let $W_*$ be the map induced by $(W,\fs_W)$ on the reduced Floer homology $\HMred(Y,\fs_Y)$ with respect to the canonical homology orientation on $W$. Then, \begin{equation*}
        |\deg(\tildeX)| =|1+2\Tr(\,\jmath_*(W_* - 1)\,)|
    \end{equation*}
    where $\Tr$ is the Lefschetz number. In particular, the above formula applies when $X$ is surgery on a 2-knot which bounds a rational homology ball. 
\end{theorem}

The majority of this paper is aimed at proving Theorem \ref{thm:deg-formula}. As an immediate corollary to Theorem \ref{thm:deg-formula}, we observe that $|\deg|$ restricts the topology of Seifert solids which a 2-knot can bound. To our knowledge, this is the first result of its type for invariants coming from real Seiberg--Witten theory. \begin{corollary}\label{cor:4d-l-space}
    If a 2-knot $\cS:S^2\hookrightarrow S^4$ bounds a punctured L-space embedded in $S^4$, then $|\deg(\cS)|=1$.
\end{corollary}

Our Theorem \ref{thm:deg-formula} is a real analogue of \cite[Theorem~A]{lin-ruberman-saveliev}, which gives a formula relating the count of ordinary Seiberg--Witten solutions on a homology $S^1\times S^3$ to the Lefschetz number of the map induced by a cobordism in a similar way. Unlike in \cite{lin-ruberman-saveliev}, our formula does not see the appearance of either a Fr\o yshov-invariant-type term nor a metric correction term. The metric correction term does not appear for the simple reason that the real irreducible is generically transverse; indeed, that this is possible is one of the main points of departure of real Seiberg--Witten theory from the ordinary theory. The Fr\o yshov term does not appear due to the form of the monopole chain complex for a $\Pin(2)$-equivariant perturbation; see Subsection \ref{subsection:the-proof}.

The first technical input to the proof of Theorem \ref{thm:deg-formula} is a chain-level definition of $\jmath_*$ that preserves the perturbation. This is dealt with in Section \ref{section:floer-homology}, where we show how Francesco Lin's construction of monopole Floer homology with $\Pin(2)$-equivariant perturbations \cite{f-lin} can be upgraded from a theory with $\Z/2$-coefficients to a theory with $\Z$-coefficients. Once this is accomplished, we can define $\jmath$ as the involution induced by $j\in\Pin(2)$.

The second technical input to the proof of Theorem \ref{thm:deg-formula} is the comparison in Section \ref{section:solution-counts} of various Seiberg--Witten moduli spaces and their orientations, based on that in \cite{lin-ruberman-saveliev}. As part of the appeal of real Seiberg--Witten theory is its apparent departure from the ordinary theory, it may seem counterintuitive to try to interpret the real invariant in terms of an ordinary one. The bridge across this gap is the relationship between real monopoles on double covers and ordinary $\spin^{c-}$ monopoles on the base.

Theorem \ref{thm:deg-formula} applies to 2-knots which bound rational homology balls. In Section \ref{section:examples}, we give some examples of 2-knots which satisfy this hypothesis and some which do not, and we compare some computations of $|\deg|$ with known results.\\\\
\noindent {\bf Acknowlegements.} Thanks to my advisor, Ciprian Manolescu, for his continual supply of encouragement and helpful insights. Thanks as well to Gary Guth, Jiakai Li, Francesco Lin, Maggie Miller, and Imogen Montague for helpful conversations.


\section{Preliminaries}\label{section:preliminaries}

\noindent In this section, we establish notation and recall the relevant definitions and properties of real $\spinc$ structures and associated invariants from Seiberg-Witten theory. 

\subsection{Real $\mathbf{spin^c}$ Structures on Double Covers}\label{subsection:real-spin-structures} Let $M$ be a smooth, oriented 3-or-4-manifold and $\iota: M\xrightarrow{\sim} M$ a smooth involution. A \emph{real} $spin^c$ structure on $(M,\iota)$ consists of a $\spinc$ structure $\fs=(S,\Gamma)$ on $M$ along with a complex anti-linear involution $I:S\to S$ such that $I$ covers $\iota$ and such that $I$ is compatible with $\Gamma$. Here $S$, respectively $\Gamma$, are the spinor bundle, respectively the Clifford multiplication, associated to $\fs$. The involution $I$ acts on sections of $S$ as well as on other associated objects. Fixed points of these actions are called \emph{real}.

Suppose that $X$ is a smooth 4-manifold with the homology of $S^1\times S^3$. There exists a unique double cover $\tildeX\to X$. The deck transformation $\iota:\tildeX\xrightarrow{\sim} \tildeX$ is a free involution. There is a unique $\spinc$ structure $\fs_{X}=(S_X,\Gamma_X)$ on $X$. Let $\fs_{\tildeX}=(S_{\tildeX},\Gamma_{\tildeX})$ be the pullback to $\tildeX$. Both $\spin$ structures on $X$ pull back to the same $\spin$ structure on $\tildeX$. This $\spin$ structure induces a quaternionic structure on the spinor bundle of $\fs_{\tildeX}$, so we can consider the action of $j\in\Pin(2)$ on spinors, $\Phi\mapsto \Phi\cdot j$. For all $x\in \tildeX$, the fibers $S_x$ and $S_{\iota(x)}$ of the spinor bundle are canonically identified with each other via their mutual identification with the corresponding fiber of $S_X$. Let $\iota^*:S_{\iota(x)}\xrightarrow{\sim} S_{x}$ be this identification. Fix $u:\tildeX\to S^1$ representing a generator of $H^1(\tildeX;\Z)\cong \Z$ such that $\iota^*u=-u$. Define $I_x:S_{x}\to S_{x}$ by \begin{equation*}
    I_x(\Phi) = (u(x)\Phi)\cdot j.
\end{equation*}
and define $I:S_{\tildeX}\to S_{\tildeX}$ by \begin{equation*}
    I(\Phi \in S_x) = \iota^*I_x \Phi = -I_{\iota(x)}(\iota^* \Phi) \in S_{\iota(x)}
\end{equation*}
The involution $I$ induces an involution on sections of $S$ by $I(\Phi)_{\iota(x)} = I(\Phi_x)$.
\begin{lemma}
    The data $(\fs_{\tildeX}, I)$ constitute a real $spin^c$ structure on $\tildeX$.
\end{lemma} \begin{proof}
    That $I$ is complex anti-linear and compatible with the Clifford multiplication is clear. Since we have taken the $\spin$ structure which is pulled back from $X$, multiplication by $j$ commutes with $\iota^*$. Using $(u\Phi)\cdot j = u^{-1}(\Phi\cdot j)$, we have \begin{equation*}
        I^2(\Phi) = (-u((u\Phi)\cdot j))\cdot j = (-u^{-1}u\Phi)\cdot j^2 = \Phi.
    \end{equation*}
\end{proof}

The real $\spinc$ structure on $\tildeX$ is unique up to isomorphism and sign \cite[Lemma~7.1]{miyazawa-park-taniguchi}. \begin{remark}
    The $\spin$ structure, on the other hand, is not unique. There are two $\spin$ structures on $\tildeX$ inducing $\fs_{\tildeX}$. We took the pullback $\spin$ structure. We could have instead taken the $\spin$ structure which is not pulled back from $X$. With respect to this other $\spin$ structure, the involution $\iota$ is of odd type, meaning it admits an order 4 lift to the principal $\spin$ bundle \cite{atiyah-bott}. Combining the map induced on spinors by the $\spin$ lift with the action of $j$ coming from the odd-type $\spin$ structure would also give a real involution $I$ \cite[\textsection~2]{konno-miyazawa-taniguchi}.
\end{remark}

\begin{remark}\label{rmk:non-unique-spin-structure}
    The involution $I$ acts on $\spinc$ connections by pullback. If $A_0$ is the $\spin$ connection for the even-type $\spin$ structure, then \begin{equation*}
            I \circ \nabla_{A_0} \circ I = \nabla_{A_0+u^{-1}du}. 
    \end{equation*}
    The connection $\tilde A_0 = A_0 + \frac 12 u^{-1}du$ is fixed under pullback along $I$, i.e., it is a \emph{real} connection. In fact, we can arrange that $\tilde A_0$ is the $\spin$ connection for the odd-type $\spin$ structure. In an imprecise sense, the two $\spin$ structures inducing $\fs$ differ by the multi-valued gauge transformations $u^{\pm 1/2}$.
\end{remark}

Consider a 2-knot $\cS:S^2\hookrightarrow S^4$. Let $X(\cS)$ be the 4-manifold obtained by surgery on $S^4$ along $\cS$. Then, $X(\cS)$ has the homology of $S^1\times, S^3$ and by the preceding discussion we get a real $\spinc$ structure on the double cover $\tildeX(\cS)\to X(\cS)$. There is an alternate way to arrive at the same real $\spinc$ structure, as follows. There is a unique real $\spin$ structure on the double branched cover $\Sigma_2(\cS)$, which induces a real $\spinc$ structure. Surgery on the double branched cover $\Sigma_2(\cS)$ along the branching locus produces $\tildeX(\cS)$, \begin{equation*}
    \tildeX(\cS) = \left(\Sigma_2(\cS)\setminus \nu(\tilde\cS) \right) \cup_{S^1\times S^2} S^1\times B^3.
\end{equation*}
The deck transformation on $\tildeX(\cS)$ is the union of involutions on $\Sigma_2(\cS)\setminus \nu(\tilde\cS)$ and $S^1\times B^3$. If we restrict the aforementioned real $\spinc$ structure on $\Sigma_2(\cS)$ to the complement $\Sigma_2(\cS)\setminus \nu(\tilde\cS)$, there is a unique way to extend it to a real $\spinc$ structure on $\tildeX(\cS)$.

\subsection{Monopoles and Real Structures}\label{subsection:monopoles-and-real-structures} {

Let $X$ be a smooth, closed, oriented 4-manifold and $\fs=(S=S^+\oplus S^-,\Gamma)$ a $\spinc$ structure on $X$. Suppose that $X$ is equipped with a Riemannian metric $g$ and let $\cC(X;\fs,g)$ be the space (completed with respect to some $L_2^k$ Sobolev norm) of pairs $(A,\Phi)$ where $A$ is a unitary $\spinc$ connection on $S$ and $\Phi$ is a smooth section of $S^+$. Let $\eta\in\Omega_2^+(X;i\R)$. The \emph{Seiberg--Witten equations} are the following pair of first-order partial differential equations \begin{equation}\label{eq:4d-seiberg-witten-equations}
    \begin{dcases}
        \frac 12 F_{A^t}^+ - \Gamma^{-1}((\Phi\Phi^*)_0) = \eta \\
        D_A^+\Phi = 0.
    \end{dcases}
\end{equation}

Let $\cN(X; \fs,g,\eta)\subseteq \cC(X;g,\fs)$ be the space of solutions to \eqref{eq:4d-seiberg-witten-equations}. The pair $(\cC,\cN)$ is acted on by the gauge group $\cG(X) = L_{k+1}^2(X,S^1)$. The \emph{configuration space} is the quotient \begin{equation*}
    \cB(X;\fs,g) = \cC / \cG
\end{equation*} 
and the \emph{Seiberg--Witten moduli space} is the quotient \begin{equation*}
    \cM(X;\fs,g,\eta) = \cN / \cG \subseteq \cB.
\end{equation*}
A pair $(A,\Phi)\in \cC$ is called \emph{reducible} if $\Phi=0$ and \emph{irreducible} otherwise. We write $\cC^*,\cB^*,\cM^*$ for the irreducible spaces. \begin{remark}
    We will typically omit parameters such as $\fs,g,\eta$, etc. from our notation when they are understood.
\end{remark}

Suppose $\iota:X\to X$ is a smooth involution such that $\iota^*g=g$, and $(\fs,I)$ is a real $\spinc$ structure on $(X,\iota)$. Then, the involution $I$ acts on the space of pairs $\cC,\cG,\cB$. Suppose the perturbation $\eta$ satisfies $\iota^*\eta=-\eta$. Then, $I$ acts on $\cN,\cM$ as well. The \emph{real configuration space} and the \emph{real moduli space} are, respectively, the quotients \begin{equation*}
    \cB_\Re(X;\fs,g,I) = \cC^I / \cG^I
\end{equation*}
and \begin{equation*}
    \cM_\Re(X;\fs,I,g,\eta) = \cN^I / \cG^I.
\end{equation*}
The real moduli space embeds as a subspace of the fixed point set of $I$ acting on the ordinary moduli space, \begin{equation*}
    \cM_\Re \hookrightarrow \cM^I.
\end{equation*}
The constant gauge transformations form a subgroup $S^1\subseteq \cG$ such that the intersection with $\cG^I$ is $\Z/2=\{\pm1\}$. Suppose we are given a splitting $\cG^I = \Z/2 \times \cG_0^I$. The \emph{framed} real moduli space is \begin{equation*}
    \cM_\Re^f = \cN^I / \cG_0^I.
\end{equation*}
The quotient map $\cM_\Re^f \to \cM_\Re$ is a bijection on reducibles and 2-to-1 on irreducibles.

Suppose $(\tildeX,\iota)$ is the double cover of a homology $S^1\times S^3$ and $(\fs_{\tildeX},I)$ is the real $\spinc$ structure described in Subsection \ref{subsection:real-spin-structures}. Then, $b_i^{-\iota^*}=0$ for $i\in\{1,2,3\}$, so we may consider the invariant $|\deg(\tildeX, \iota, \fs_{\tildeX}, I)|$ defined in \cite{miyazawa}. For a generic choice of $\iota$-anti-invariant perturbation $\eta$, the real framed moduli space $\cM_\Re^f(\tildeX)$ is a smooth 0-dimensional manifold oriented up to a global ambiguity of sign, and \begin{equation*}
    |\deg(\tildeX)| = |\#\cM_\Re^f(\tildeX)|.
\end{equation*}
When $\tildeX = \tildeX(\cS)$ is the double-cover of surgery on a 2-knot in $S^4$, then $|\deg(\tildeX)| = |\deg(\cS)|$, where $|\deg(\cS)|$ is the invariant for the double branched cover of $\cS$ with its $\spin$ structure; cf. \cite[Proposition~7.3]{miyazawa-park-taniguchi}.

\begin{remark}\label{rmk:pin2-structures}
    Monopole invariants coming from real $\spinc$ structures on double covers were studied in \cite{n-nakamura} from the point of view of \emph{$spin^{c-}$ structures}. We may adopt this point of view in the following way. The left quotient $I \setminus S_{\tildeX}$ is an $\R^4\oplus \R^4$-bundle over $X$ with structure group $\mathrm{Spin}^{c-}(4) = \Spin(4)\times_{\Z/2}\Pin(2)$. Real $\spinc$ objects associated to $S_{\tildeX} , I$ (e.g., real spinors, real $\spinc$ connections, real gauge transformations, Clifford multiplciation by $\iota^*$-anti-invariant forms)  descend to $\spin^{c-}$ objects associated to the quotient. We write $I\setminus \fs_{\tildeX}$ for the $\spin^{c-}$ structure over $X$. There is a canonical homeomorphism between the real moduli space $\cM_\Re(\tildeX,\iota,\fs_{\tildeX},I)$ and the $\spin^{c-}$ moduli space $\cM(X,I\setminus \fs_{\tildeX})$ which restricts to a diffeomorphism of the irreducible spaces.
\end{remark}

We will sometimes need to work with the fixed-point set $\cM(\tildeX)^I$ when we really want to work with the real moduli space $\cM_\Re(\tildeX)$. Recall that the involution $I$ was defined in Subsection \ref{subsection:real-spin-structures} by combining the action of $j$ with the gauge transformation $u$. We record the following lemmas for future reference. \begin{lemma}\label{lem:purely-imaginary-implies-zero}
    Suppose that a spinor $\Phi$ satisfies $I(\Phi)= u\Phi$. Then, $\Phi \equiv 0$.
\end{lemma}\begin{proof}
    Such a spinor satisfies\begin{equation}\label{eq:I-equals-u}
        \begin{dcases}
            u(x)\Phi_x = (u(\iota(x))\Phi_{\iota(x)})\cdot j \\
            u(\iota(x))\Phi_{\iota(x)} = (u(x)\Phi_x)\cdot j.
        \end{dcases}
    \end{equation}
    From \eqref{eq:I-equals-u}, we deduce \begin{equation*}
        u(x)\Phi_x = ((u(x)\Phi_x)\cdot j)\cdot j = -u(x)\Phi_x.
    \end{equation*}
    Therefore, $\Phi\equiv 0$.
\end{proof}\begin{lemma}\label{lem:irreducible-moduli-spaces}
    The embedding $\cM_\Re(\tildeX)\hookrightarrow \cM(\tildeX)^I$ restricts to a bijection $\cM_\Re^*(\tildeX)\hookrightarrow \cM^*(\tildeX)^I$.
\end{lemma}\begin{proof}
    The fixed-point set $\cM^*(\tildeX)^I$ consists of gauge-equivalence classes $[A,\Phi]$ such that $(A,\Phi),I(A,\Phi)$ are gauge-equivalent. Consider such a configuration. Up to a gauge transformation, we may assume $I(A,\Phi)=u^d\cdot(A,\Phi)$ for some $d\in\Z$. If $d$ is even, then $(A',\Phi')=u^{d/2}(A,\Phi)$ is gauge-equivalent to $(A,\Phi)$ and is fixed by $I$, so that $[A,\Phi]$ belongs to the real moduli space. If $d$ is odd, then $(A',\Phi') = u^{(d-1)/2}(A,\Phi)$ is gauge-equivalent to $(A,\Phi)$ and satisfies $I(A',\Phi')=u(A',\Phi')$, in which case Lemma \ref{lem:purely-imaginary-implies-zero} implies that $(A,\Phi)$ is reducible.
\end{proof}


}


\section{Monopole Floer Homology with Pin(2)-Equivariant Perturbations}\label{section:floer-homology}

\noindent In this section we define monopole Floer homology with $\Z$ coefficients using $\Pin(2)$-equivariant perturbations by adding appropriate orientations to the characteristic 2 construction of \cite{f-lin}. This will allow us to work with a chain-level lift of the complex conjugation action $\jmath$ on monopole Floer homology. Throughout, $Y$ will be a 3-manifold and $\fs_Y$ a $\spinc$ structure on $Y$. Additional hypotheses will be added as appropriate.

\subsection{Perturbations and Moduli Spaces}\label{subsection:perturbations-and-moduli-spaces}{

Let $Y$ be an oriented 3-manifold and $\fs_Y$ a $\spinc$ structure on $Y$. The construction of monopole Floer homology \cite{kronheimer-mrowka} requires a choice of Riemannian metric $g$ and a perturbation $\fq$ of the Chern-Simons-Dirac functional $\cL$. The construction in \cite{f-lin} considers perturbations such that the perturbed gradient $\nabla\pounds$ has Morse-Bott singularities \cite[ch.~2, Definition~1.2]{f-lin}. This allows for more general perturbations than those considered in \cite{kronheimer-mrowka}, for which the critical points are required to be genuinely non-degenerate.

Suppose that a perturbation has been chosen such that all singularities are Morse-Bott. We will refer to connected components of the set of critical points as \emph{critical submanifolds}. Let $[\fC^-],[\fC^+]$ be two critical submanifolds. The moduli space  $\cM([\fC^-],[\fC^+])$
is the space of trajectories of $\nabla\pounds$ limiting at the ends to points in $[\fC^\pm]$. The moduli space $\cM([\fC^-],[\fC^+])$ comes with an action of $\R$ by time-translation and we write $\ucM([\fC^-],[\fC^+])$ for the quotient, the moduli space of unparameterized trajectories. The moduli space $\cM([\fC^-],[\fC^+])$ is defined as a subspace of an appropriate completion of the blown-up configuration space of a cylinder $\R\times Y$. It may be considered as a subspace of either the Sobolev-completed configuration space $\cB_{k}^{\tau}(\R\times Y)$ or the weighted-Sobolev-completed configuration space $\cB_{k,\delta}^{\tau}(\R\times Y)$; see \cite[pp. 43-44]{f-lin} for the appropriate notion of weighted Sobolev completion. The weighted spaces are introduced to deal with Fredholm issues at the ends.

The appropriate notion of regularity of the moduli space is \cite[Definition~3.14]{f-lin}. Regularity requires transversality of the evaluation maps \begin{equation*}
    \ev^\pm : \ucM([\fC^-],[\fC^+]) \longrightarrow [\fC^\pm]
\end{equation*} 
sending a trajectory to its endpoints, and, inductively, transversality such that the iterated fiber products \begin{equation*}\begin{split}
    &\ucM([\fC^-]=[\fC_1],[\fC_2],\cdots,[\fC_n]=[\fC^+]) = \\ &\lim\left(\begin{tikzcd}[ampersand replacement=\&,sep=small]
    \ucM([\fC_1],[\fC_2]) \ar[dr] \& \& \ar[dl] \ucM([\fC_2],[\fC_3]) \ar[dr] \& \cdots \& \ar[dl] \ucM([\fC_{n-1}],[\fC_n]) \\
    \& {[\fC_2]} \& \& \cdots
\end{tikzcd}\right),
\end{split}
\end{equation*}
the moduli spaces of broken trajectories, inherit the structure of smooth manifolds. The union of moduli spaces of broken trajectories gives a compactification of $\ucM([\fC^-],[\fC^+])$, \begin{equation*}
    \ucM^+([\fC^-],[\fC^+]) = \coprod \ucM([\fC^-]=[\fC_1],[\fC_2],\cdots,[\fC_n]=[\fC^+]).
\end{equation*}

We also consider moduli spaces of Seiberg--Witten solutions on cobordisms. Let $W$ be an oriented compact $\spinc$ 4-manifold with boundary $\partial W = \coprod_{\alpha} Y^\alpha$ and let $W_\infty$ be the open 4-manifold obtained from $W$ by attaching cylindrical ends. Suppose that a perturbation $\fp$ is chosen over $W$. We write $\cM(W_\infty; ([\fC^\alpha]))$ for the moduli space of Seiberg--Witten solutions limiting to points in $[\fC^\alpha]$ over each end $Y^\alpha$ and $\cM^+(W_\infty; ([\fC^\alpha]))$ for the compactified moduli space. Sometimes, we will view some of the boundary components of $W$ as ``incoming''. Suppose the oriented boundary of $W$ is $\partial W = (-\coprod_\beta Y^{-,\beta}) \sqcup (\coprod_\alpha Y^{+,\alpha})$. Then, we write $\cM(([\fC^\beta]), W_\infty, ([\fC^\alpha]))$ for the moduli space of solutions that limit to points in $[\fC^\beta]$ over $Y^{-,\beta}$ and to points in $[\fC^\alpha]$ over $Y^{+,\alpha}$, and $\cM^+(([\fC^\beta]), W_\infty, ([\fC^\alpha]))$ for the compactified moduli space.

When a perturbation is chosen (over a 3-manifold $Y$ or over a cobordism $W$) such that all moduli spaces involved are regular, we will say that the perturbation is regular. For a regular perturbation, the compactified moduli spaces are stratified by manifolds and have codimension-$c$ $\delta$-structures \cite[24.7.1]{kronheimer-mrowka} along each face of the codimension-1 stratum; see \cite[ch.~2, \textsection\textsection~5-6]{f-lin}. Here, $c$ is the order to which the corresponding face is boundary-obstructed.

Suppose that $\fs_Y$ is induced by a $\spin$ structure. Then, the quaternion $j\in\Pin(2)$ acts on the spinor bundle, inducing an involution $\jmath$ on the configuration space. We consider perturbations which are equivariant with respect to the pushforward along $\jmath$, and hence with respect to the action of $\Pin(2)$. It is generally not possible to find $\Pin(2)$-equivariant perturbations which are genuienly non-degenerate. Instead, the appropriate notion is $\Pin(2)$-non-degeneracy \cite[ch.~4, Definition~2.3]{f-lin}. For a $\Pin(2)$-non-degenerate perturbation, all critical points are non-degenerate in the usual sense, except for reducible critical points lying over a $\spin$ connection, which come in Morse-Bott $\CP^1$ families. Regularity of $\Pin(2)$-equivariant perturbations is dealt with in \cite[ch.~4, Theorem~2.6]{f-lin} and \cite[ch.~4, Proposition~2.12]{f-lin}.

}

\subsection{Orienting Abstract $\delta$-chains}\label{subsection:orienting-abstract-delta-chains} The Morse-Bott monopole Floer complex is built from chains modeled on moduli spaces of Seiberg--Witten solutions. In order to define homology, we need to understand the combinatorial properties of these chains and how orientations propagate down the hierarchy of faces. \begin{definition} (cf. \cite[ch.~3, Definition~1.1]{f-lin})
    An \emph{abstract $\delta$-chain} is a stratified space \begin{equation*}
        \Delta = \left(N^d\supseteq N^{d-1} \supseteq \cdots \supseteq N^{1}\supseteq N^0 \supseteq \emptyset\right)
    \end{equation*}
    along with the following data. Each stratum has a partition \begin{equation*}
        N^e = \coprod_{a} M^e_{a}
    \end{equation*} 
    and the partitions respect the hierarchy of inclusions. Closures of partition elements are called \emph{faces}. If $\Sigma=\overline{\Delta^e_a}$ is a face, then $\Sigma^\circ = \Delta^e_a$ denotes its top-dimensional stratum.

    Whenever a codimension-$e$ face is contained in a codimension-$(e-2)$ face, there are exactly two codimension-$(e-1)$ faces between them (in particular, every codimension-2 face is contained in exactly two codimension-1 faces).

    Associated to every inclusion of faces $\Sigma_{a}\subseteq \Sigma$, there are \begin{enumerate}
        \item a pair of finite sets $O_{\Sigma,\Sigma_{a}}\subseteq N_{\Sigma,\Sigma_{a}}$ called the \emph{obstruction} and \emph{parameter} sets, respectively,
        \item an open neighborhood $\breve{W}_{\Sigma,\Sigma_{a}}\supseteq \Sigma_{a}$ in $\Sigma$ and a topological embedding $\breve W \hookrightarrow E\breve W_{\Sigma,\Sigma_{a}}$ into a larger space (the \emph{local thickening}), and
        \item continuous maps $\bfS=\bfS_{\Sigma,\Sigma_{a}}:\EW\to(0,\infty]^{N_{\Sigma,\Sigma_{a}}}$ and $\delta=\delta_{\Sigma,\Sigma_{a}}:\EW\to\R^{O_{\Sigma,\Sigma_{a}}}$
    \end{enumerate}
    satisfying the following properties. \begin{enumerate}
        \item The subspaces ${\Sigma_{a}}^\circ,\W\subseteq\EW$ are identified as fibers of $\bfS,\delta$, respectively. Namely, ${\Sigma_{a}}^\circ=\bfS^{-1}(\vinfty)$ and $\W = \delta^{-1}(0)$.
        \item The map $\bfS$ is a topological submersion along $\W$ (\cite[Definition~19.2.7]{kronheimer-mrowka}).
        \item Where all components of $\bfS$ are finite, the embedding $\W\hookrightarrow\EW$ restricts to a smooth embedding of manifolds and $\delta$ is smooth and transverse to $0$.
    \end{enumerate}
    The data above is required to be compatible with respect to the hierarchy of inclusions, i.e., a chain of inclusions of faces $\Sigma_{ab}\subseteq \Sigma_{a}\subseteq \Sigma$ induces inclusions $N_{\Sigma_{a},\Sigma_{ab}}\subseteq N_{\Sigma,\Sigma_{ab}}$ and $O_{\Sigma_{a},\Sigma_{ab}} \subseteq O_{\Sigma,\Sigma_{ab}}\cap N_{\Sigma_{a},\Sigma_{ab}}$, we have identifications \begin{equation*}
        \begin{split}
            (0,\infty]^{N_{\Sigma_{a},\Sigma_{ab}}} &= \left\{(x_a)\in (0,\infty]^{N_{\Sigma,\Sigma_{ab}}}\,:\,x_a=\infty\text{ unless }a\in N_{\Sigma_{a},\Sigma_{ab}}\right\} \\
            \EW_{\Sigma_{a},\Sigma_{ab}} &= \bfS^{-1}_{\Sigma,\Sigma_{ab}}\left((0,\infty]^{N_{\Sigma_{a},\Sigma_{ab}}}\right)\\
            \bfS_{\Sigma_{a},\Sigma_{ab}} &= \bfS_{\Sigma,\Sigma_{ab}}|_{\EW_{\Sigma_{a},\Sigma_{ab}}}\\
            \delta_{\Sigma_{a},\Sigma_{ab}} &= \delta_{\Sigma,\Sigma_{ab}}|_{\EW_{\Sigma_{a},\Sigma_{ab}}}.
        \end{split}
    \end{equation*}
    The identification of the local thickenings need only be defined on a neighborhood of the fiber over $\vinfty$. Whenever $\mathrm{codim}_{\Sigma} \Sigma_{a}=1$, there is an identification of $\R^{O_{\Sigma,\Sigma_a}}$ with the hyperplane $\Pi^c=\vec 1^\perp\subseteq \R^{c+1}$ and the local thickening and the pair of maps $\bfS, \delta$ give $\Sigma$ a codimension-$c$ $\delta$ structure along $\Sigma_{a}$ (\cite[Definition~24.7.1]{kronheimer-mrowka}), where $c=c_{\Sigma,\Sigma_a}:=|O_{\Sigma,\Sigma_a}|\equiv |N_{\Sigma,\Sigma_a}|-1$. 
\end{definition}

In order to define a $\Z$-chain complex based on $\delta$-chains, we need to be able to orient our $\delta$-chains. The combinatorial structure of $\delta$-chains in general is too weak to make sense of orientations in a coherent way, and so we restrict attention to a subclass. The combinatorial structure of these chains is more strictly modeled on that of the Seiberg--Witten moduli spaces we will encounter, following \cite[Section 24.6]{kronheimer-mrowka}.\begin{definition}\label{def:delta-chain-properties}
    An abstract $\delta$-chain $\Delta$ is \emph{orientable} if $\Delta^\circ$ is orientable as a manifold and if for every inclusion of faces $\Sigma_a\subseteq\Sigma$ with $\operatorname{codim}_\Sigma \Sigma_a = 1$, there are preferred orientations on $\R^{O_{\Sigma,\Sigma_a}}$ and $(0,\infty]^{N_{\Sigma,\Sigma_{a}}}$. An \emph{orientation} of $\Delta$ is an orientation of $\Delta^\circ$. If $\Delta$ is orientable and is given an orientation then $\Delta$ is \emph{oriented}.
\end{definition}

\begin{example}[Moduli Spaces of Trajectories as Abstract $\delta$-chains]\label{ex:cylinder}
    Consider a moduli space of unparameterized Seiberg--Witten trajectories \begin{equation*}
        \Delta=\breve\cM^+([\fC^{-}],[\fC^{+}]).
    \end{equation*}
    If $\Delta$ contains irreducibles, then there is a classification of the codimension-1 faces analogous to \cite[Proposition~16.5.5]{kronheimer-mrowka}. There are\begin{enumerate}[label=(\roman*)] 
        \item moduli spaces of boundary-unobstructed broken trajectories, \begin{equation*}
            \Sigma = \ucM^+([\fC^{-}], [\fA], [\fC^{+}]),
        \end{equation*}
        for which $N_{\Delta,\Sigma} = \{[\fA]\}$ and $O_{\Delta,\Sigma} = \emptyset$.
        \item moduli spaces of broken trajectories with a boundary-obstructed component, \begin{equation*}
            \Sigma = \ucM^+([\fC^{-}], [\fA],  [\fB], [\fC^{+}]),
        \end{equation*}
        for which $N_{\Delta, \Sigma} = \{[\fA],[\fB]\}$ and $O_{\Delta,\Sigma} = \{[\fB]\}$.
        \item in the case that the moduli space contains both irreducibles and reducibles, the reducible part \begin{equation*}
            \Sigma = \ucM^{red+}([\fC^{-}],[\fC^{+}]),
        \end{equation*}
        for which $N_{\Delta,\Sigma} = \{[\fC^{+}]\}$ and $O_{\Delta,\Sigma} = \emptyset$.
    \end{enumerate}
    In the first and third cases above, $\Delta$ has the structure of a manifold with boundary along $\Sigma$. The local thickening $\EW$ is just a neighborhood $\W$ and $\bfS$ parameterizes $\W$ as a collar $\Sigma\times (0,\infty]$. In the second case, we have a codimension-1 $\delta$-structure. The local thickening $\EW$ consists of trajectories for which the size of the spinor component is allowed to have a discontinuity at zero. The map $\delta:\EW\to \R$ measures this discontinuity, while $\bfS$ records breaking of trajectories. See \cite[ch.~19]{kronheimer-mrowka}.
    
    In the case that $\Delta$ consists entirely of reducibles, the codimension 1 faces are all of the form \begin{equation*}
        \Sigma = \ucM^{red+}([\fC_0]=[\fC^{-}], [\fC_1], [\fC_2]=[\fC^{+}]),
    \end{equation*}
    for which $N_{\Delta,\Sigma} = \{[\fC_1]\}$ and $O_{\Delta,\Sigma}=\emptyset$.
    
    For a general inclusion of faces $\Sigma_a\subseteq\Sigma$, the set $N_{\Sigma,\Sigma_a}$ records breaking of solutions, and $O_{\Sigma,\Sigma_a}$ records boundary-obstructed breaking. The set $N_{\Sigma,\Sigma_a}$ is ordered in the obvious way. The ordering on $N_{\Delta,\Sigma}$ gives preferred orientations on $(0,\infty]^{N_{\Sigma,\Sigma_{a}}}$ and $\R^{O_{\Sigma,\Sigma_{a}}}$ for every codimension-1 inclusion of faces $\Sigma_a\subseteq\Sigma$.
\end{example}
\begin{remark}
    We have modified the definition of an abstract $\delta$-chain slightly compared to \cite[ch.~3, Definition~1.1]{f-lin}. Specifically, for a chain of inclusions of faces $\Sigma_{ab}\subseteq \Sigma_{a}\subseteq \Sigma$, we do not require equality $O_{\Sigma_{a},\Sigma_{ab}} = O_{\Sigma,\Sigma_{ab}}\cap N_{\Sigma_{a},\Sigma_{ab}}$. This becomes relevant, for example, if we have a chain of moduli spaces of broken trajectories of the form \begin{equation*}
        \begin{split}
            &\Sigma_{ab} = \ucM^+([\fC_0],[\fC_1])\times \ucM^+([\fC_1],[\fC_2])\times \ucM^+([\fC_2],[\fC_3]) \times \ucM^+([\fC_3],[\fC_4]) \\
            &\subseteq \Sigma_{a} = \ucM^+([\fC_0],[\fC_1])\times \ucM^+([\fC_1],[\fC_3])\times \ucM^+([\fC_3],[\fC_4]) \\
            &\subseteq \Sigma = \ucM^+([\fC_0],[\fC_4])
        \end{split}
    \end{equation*}
    where $[\fC_0],[\fC_4]$ are irreducible, $[\fC_1]$ is boundary-stable, and $[\fC_2],[\fC_3]$ are boundary-unstable. In this case, $O_{\Sigma,\Sigma_{ab}} = \{[\fC_2]\} = N_{\Sigma_a,\Sigma_{ab}}$, but $O_{\Sigma_a,\Sigma_{ab}}=\emptyset$. Even though $\Sigma_{ab}$ contains broken trajectories with boundary-obstructed components, we do not record them as part of $O_{\Sigma_a,\Sigma_{ab}}$ as they occur as part of the boundary of a reducible moduli space.
\end{remark}

\begin{definition}\label{def:boundary-orientation}
    Let $\Delta$ be an oriented abstract $\delta$-chain and $\Sigma\subseteq\Delta$ a codimension-1 face. The \emph{boundary orientation} on $\Sigma$ is that orientation which differs from the orientation described above by the sign $(-1)^{\dim\Delta + c_{\Delta,\Sigma}+1}$.
\end{definition}
For the simplest example of a codimension-$c$ $\delta$-structure, take $\EW=\R^d\times (0,\infty]^{c+1}$ with the maps $\bfS(\mathbf x,\mathbf s)=\mathbf s$ and $\delta(\mathbf x,\mathbf s) = (1/ s_2 - 1 /s_2, \ldots , 1 / s_{c+1} - 1 / s_c)$. In this case, $\W = \delta^{-1}(0) = \R^d\times \mathrm{span}(1,\ldots,1)$ has the structure of a manifold with boundary along $\R^d\times \{0\}$. The boundary orientation is defined so that in this case it agrees with the orientation obtained using the outward-normal-first convention.

A codimension-2 face of an oriented abstract $\delta$-chain is contained in two codimension-1 faces, and can therefore inherit two possible boundary orientations. In order to define homology based on $\delta$-chains, we need these orientations to cancel.

If $N^d\supseteq N^{d-1}\supseteq \ldots$ is a space stratified by manifolds and each stratum is oriented, then the \v{C}ech cohomology $\check{H}^e(N^e,N^{e-1};\Z)$ is canonically isomorphic to $\Z^n$, with a basis given by the components of $N^e\setminus N^{e-1}$. There is a sequence of coboundary operators $\delta_*:\check{H}^{e-1}(N^{e-1},N^{e-2};\Z) \to \check{H}^e(N^e,N^{e-1};\Z)$ satisfying $\delta_*^2=0$. If $M^e_\alpha,M^{e-1}_\beta$ are components of $N^e,N^{e-1}$ respectively, then the \emph{boundary multiplicity} $\delta_{\alpha\beta}$ of $M^{e-1}_{\alpha}$ in $M^{e}_{\beta}$ is the coefficient of $[M^{e}]$ in $\delta_*[M^{e-1}]$, and the boundary multiplicity of $M^{e-1}_\alpha$ in $N^e$ is the sum $\delta_{\alpha;N} = \sum_\beta \delta_{\alpha\beta}$

\begin{lemma}\label{lem:canceled}
    Let $\Delta$ be an abstract $\delta$-chain. Suppose $\Delta$ is oriented and the codimension-1 faces of $\Delta$ are given their boundary orientations. Then, for every square 
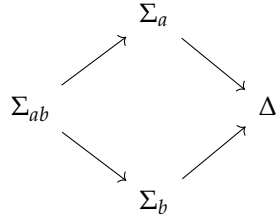
\begin{equation*}
            \begin{tikzcd}
                & \Sigma_{a}\ar[dr] & \\
                \Sigma_{ab}\ar[ur]\ar[dr] & & \Delta \\
                & \Sigma_{b}\ar[ur] &
            \end{tikzcd}
        \end{equation*}
        of codimension-1 inclusions of faces with $\Sigma_{a}\neq \Sigma_{b}$, the boundary orientations on $\Sigma_{ab}$ induced through $\Sigma_a,\Sigma_b$ are opposite.
\end{lemma} \begin{proof}
    Consider such a square of inclusions. Suppose that $\Sigma_{ab}$ is given the boundary orientation induced through $\Sigma_{a}$. By \cite[Lemma~21.3.1]{kronheimer-mrowka} (slightly generalized to the codimension-$c$ case), every component of $\Sigma_{ab}$ has boundary multiplicity 1 in $\Sigma_{a}$, and $\Sigma_{ab}$ has the boundary orientation in $\Sigma_{b}$ if and only if the same conclusion holds for $\Sigma_{b}$. Let $C$ be the set of components of $\Sigma_{ab}$. We have \begin{equation*}
        0 = \delta_*^2[\Sigma_{ab}] = \sum_{\alpha\in C}\left(\delta_{\alpha;\Sigma_a} + \delta_{\alpha;\Sigma_b}\right) [\Delta].
    \end{equation*}
    Since $\delta_{\alpha\Sigma_a}=1$ for all $\alpha\in C$, we must have $\delta_{\alpha\Sigma_b}=-1$ for all $\alpha\in C$. Therefore, $\Sigma_{ab}$ has the opposite of the boundary orientation in $\Sigma_b$.
\end{proof}

\subsection{Homology of Smooth Manifolds via $\delta$-chains}\label{subsection:homology-via-delta-chains} Following \cite{lipyanskiy}, we can define a version of singular homology for smooth manifolds where the chains are maps from abstract $\delta$-chains. Let $M$ be a smooth manifold. A \emph{$\delta$-chain in $M$} consists of an abstract $\delta$-chain $\Delta$ and a continuous map $\sigma:\Delta\to M$ which is smooth on strata, and for every inclusion of faces $\Sigma_a\subseteq\Sigma$ an extension $E\sigma:\EW_{\Sigma,\Sigma_a}\to M$ which is smooth near the fiber over $\vinfty$, such that the extensions respect the hierarchy of inclusions (cf. \cite[ch.~3,~Definition~1.6]{f-lin}). \begin{remark}
    We will write of the set of $\delta$-chains in $M$, by which we really mean the set of isomorphism classes of $\delta$-chains in $M$ where the underlying abstract $\delta$-chains are all embedded in some fixed infinite-dimensional Hilbert space.
\end{remark} \begin{example}
    If $[\fC^-],[\fC^+]$ are Morse-Bott critical sumbanifolds and $\ucM^+([\fC^-],[\fC^+])$ is regular, then the evaluation maps \begin{equation*}
        \ev^{\pm} : \ucM^+([\fC^-],[\fC^+]) \longrightarrow [\fC^{\pm}]
    \end{equation*}
    defined by sending a trajectory to its limit points are $\delta$-chains.
\end{example}

A $\delta$-chain $\sigma$ is oriented if $\Delta$ is oriented, in which case $-\sigma$ denotes the same $\delta$-chain with the orientation on $\Delta$ reversed. Addition of $\delta$-chains is defined using the disjoint union of abstract $\delta$-chains, \begin{equation*}
    (\sigma,\Delta)+(\sigma',\Delta') = (\sigma\sqcup \sigma', \Delta\sqcup\Delta').
\end{equation*} 
This makes the set of oriented $\delta$-chains in $M$ into an abelian monoid. Let $EC_*^\delta(M;\Z)=EC^\delta_*(M)$ denote the quotient of this monoid by the relation $\sigma+(-\sigma)=0$. Then, $EC^\delta_*(M)$ is an abelian group graded by the dimensions of the chains. The boundary operator $\partial:EC^\delta_*(M)\to EC^\delta_{*-1}(M)$ is defined as the sum of the codimension-1 faces, \begin{equation*}
    \sigma=\sum_{\operatorname{codim}_\Delta \Sigma = 1} \sigma|_\Sigma.
\end{equation*}
The codimension-1 faces are understood to have the boundary orientations. \begin{lemma}
    The boundary operator $\partial$ makes $EC^\delta_{*}(M)$ into a chain complex.
\end{lemma}\begin{proof}
    By Lemma \ref{lem:canceled}, $\partial^2=0$.
\end{proof}

At this point, we have a chain complex $(EC^\delta_*(M),\partial)$ but the resulting homology is not the singular homology of $M$. In order for our homology theory based on $\delta$-chains to satisfy appropriate versions of the Eilenberg-Steenrod axioms, we need to eliminate certain trivial and degenerate chains. \begin{definition}
    An oriented $\delta$-chain $(\sigma,\Delta)$ is \emph{trivial} if there is an orientation-reversing automorphism of $\Delta$ with respect to which $\sigma$ and the extensions $E\sigma$ are invariant.
\end{definition}
\begin{definition}
    A $\delta$-chain $(\sigma,\Delta)$ \emph{has small image} or is \emph{small} if there exists a $\delta$-chain $(\sigma',\Delta')$ with $\dim\Delta'<\dim \Delta$ such that the image of $\sigma$ is contained in that of $\sigma'$.
\end{definition}
For example, any $\delta$-chain $(\sigma,\Delta)$ with $\dim\Delta>\dim M$ is small. \begin{definition}
    An oriented $\delta$-chain $\sigma$ is \emph{degenerate} if $\sigma$ is small and $\partial \sigma=\alpha+\beta$ where $\alpha$ is trivial and $\beta$ is small. 
\end{definition}
\begin{definition}
    The \emph{singular $\delta$-chain group of $M$ (with $\Z$-coefficients)} is the graded abelian group \begin{equation*}
        C^\delta_{*}(M) = \frac{EC^\delta_*(M)}{\{[\sigma]+[\sigma']:\sigma\text{ is trivial and }\sigma'\text{ is degenerate}\}}.
    \end{equation*}
\end{definition}
If $\dim\sigma > \dim M + 1$, then $\sigma$ is degenerate. For future reference, we record the following lemma, which implies that 1-dimensional $\delta$-chains in a 0-dimensional manifold are degenerate. \begin{lemma}\label{lem:degenerate-1-chains}
    If $\dim\sigma=1$, then $\partial\sigma$ is trivial.
\end{lemma}\begin{proof}
    This follows immediately from \cite[Corollary~21.3.2]{kronheimer-mrowka}, generalized to codimension-$c$ $\delta$-structures.
\end{proof}

\begin{proposition}\label{prop:delta-chain-homology}
    The boundary operator descends to a well-defined homomorphism $\partial: C^\delta_{*}(M)\to C^\delta_{*-1}(M)$ making $C^\delta_{*}(M)$ into a chain complex. The homology $H^\delta_*(M;\Z) = H_*(C^\delta_{*}(M),\partial)$ is naturally (with respect to smooth maps) isomorphic to the singular homology $H_*(M;\Z)$.
\end{proposition} \begin{proof}
    This follows from the same argument used in \cite[Section~6]{lipyanskiy}.
\end{proof}

We will use the term \emph{$\delta$-chain} to refer both to $\delta$-chains and elements of $C^\delta_{*}(M)$. In \cite[ch.~3,~Section~1]{f-lin}, it is shown how the fiber product of two transverse $\delta$-chains itself has the structure of a $\delta$-chain. If $M$ is oriented, then there is an induced orientation on the fiber product of transverse oriented $\delta$-chains. We follow the convention that, if $f:A\to M,g:B\to M$ are transverse smooth maps of oriented manifolds, then the fiber product $f\times_M g$ is oriented as the preimage of the diagonal under $f\times g:A\times B\to M\times M$. \begin{lemma}\label{lem:fiber-product-boundary-orientation}
    Let $\sigma,\sigma'$ be transverse oriented $\delta$-chains in $M$. If $M$ is oriented, $\partial \sigma,\partial\sigma'$ are given the boundary orientations, and $\sigma\times_M\sigma',(\partial\sigma)\times_M\sigma',\sigma\times_M(\partial\sigma')$ are given the fiber product orientations, then the oriented boundary of $\sigma\times_M\sigma'$ is \begin{equation*}
        \partial(\sigma\times_M \sigma') = (-1)^{\dim M}\left((\partial \sigma)\times_M \sigma' + (-1)^{\dim\sigma}\sigma\times_M (\partial \sigma')\right).
    \end{equation*}
\end{lemma} \begin{proof}
    In the case that $\sigma,\sigma'$ are smooth maps of oriented manifolds-with-boundary or manifolds-with-codimension-2-corners, the lemma is straightforward to prove by commuting normal bundles past each other. In the general case, an argument analogous to the proof of \cite[Lemma~21.3.1]{kronheimer-mrowka} can be used to pass to the manifold-with-corners case.
\end{proof}

If $\cF$ is a countable family of $\delta$-chains in $M$, we can replace the full complex $C^\delta_*(M)$ with the subcomplex $C^{\delta;\cF}_*(M)$ consisting of $\delta$-chains which are transverse to all members of $\cF$, and the inclusion $C^{\delta;\cF}_*(M)\to C^\delta_*(M)$ is a quasi-isomorphism. Sometimes, we will omit $\cF$ from our notation.

\subsection{Orienting Moduli Spaces of Trajectories}\label{subsection:orienting-moduli-spaces}
As in \cite[Section~20]{kronheimer-mrowka}, moduli spaces of trajectories will be oriented relative to orientation sets at the endpoints. Let $\fa^-,,\fa^+\in\cC^{\sigma}_{k,\delta}(Y)$ and let $\fq_1,\fq_2$ be perturbations. Fix a compact interval $I=[t_1,t_2]$. Let $\gamma_0\in\cC^{\tau}_{k,\delta}(I\times Y)$ be a configuration with boundary restrictions $\gamma|_{\{t_i\}\times Y} = \fa_i$ for $i=1,2$ and let $\fp$ be a path in a large Banach space of perturbations with $\fp(t_i)=\fq_i$ for $i=1,2$. Consider the operator \begin{equation*}
    Q_{\gamma_0,\fp} =\fF_{\fq} \oplus \bfd^{\tau,\dagger}_{\gamma_0}: \cT^{\tau}_{1,\delta,\gamma}(I\times Y) \to \cV^{\tau}_{0,\delta,\gamma_0}(I\times Y)\oplus L^{2}(Y;i\R).
\end{equation*}
encoding the linearized Seiberg--Witten equations and gauge-fixing condition. By augmenting $Q_{\gamma_0,\fp}$ with boundary conditions as in \cite[Section~20.3]{kronheimer-mrowka}, we obtain a Fredholm operator $P_{\gamma_0,\fp}$. The determinant line bundle $\det(P_{\gamma_0,\fp})$ descends to a real line bundle over $\mathcal{B}^{\tau}_{k,\delta}(I\times Y;[\fa^-],[\fa^+])$. For $t_1$ sufficiently negative and $t_2$ sufficiently positive, the orientation set $\Lambda([\fa^-],\fq_1,[\fa^+],\fq_2)$ is identified with the set of local orientations of this line bundle over $[\gamma]$. There are also orientation sets $\Lambda([\fa],\fq)$ such that when $[\fa_1],[\fa_2]$ are critical points there is a canonical bijection \begin{equation*}
    \Lambda([\fa^-],\fq_1,[\fa^+],\fq_2) = \Lambda([\fa_1],\fq_1)\times_{\Z/2}\Lambda([\fa_2],\fq_2).
\end{equation*} 
When a perturbation is fixed, we write $\Lambda([\fa^-],[\fa^+])$ and $\Lambda([\fa])$.

Fix a tame perturbation $\fq$ such that all critical points are Morse-Bott and all moduli spaces of trajectories are regular. Let $[\fC^-],[\fC^+]\subseteq \cB^{\sigma}_{k,\delta}(Y)$ be critical submanifolds and let $[\fU_i]\subseteq [\fC_i]$ be contractible open sets with smooth lifts $\fU_i\subseteq \cC^{\sigma}_{k,\delta}(Y)$. Consider a solution $[\gamma]\in\ucM([\fU^-],[\fU^+])$ limiting to $[\fa^-]\in[\fU^-], [\fa^+]\in[\fU^+]$. On the infinite cylinder, the operator $P_{\gamma}=Q_{\gamma,\fq}$ is already Fredholm, without introducing boundary conditions. There is a canonical isomorphism \begin{equation*}
    \det(P_{\gamma}) = \det(P_{\gamma_{-}}) \otimes \det(P_{\gamma_{0}}) \otimes \det(P_{\gamma_{+}}) 
\end{equation*}
where $P_{\gamma_{-}}$, $P_{\gamma_{0}}$, $P_{\gamma_{+}}$ are Fredholm operators defined by adding appropriate boundary conditions to $Q_{\gamma,\fq}$ over the intervals $(-\infty,t_1]$, $[t_1,t_2]$, $[t_2,+\infty)$, respectively. Suppose $t_1$ is sufficiently negative and $t_2$ is sufficiently positive. The orientation set $\Lambda([\fa^-],[\fa^+])$ is identified with the set of orientations of the line $\det(P_{\gamma_0})$. In the Morse-Bott setting, $P_{\gamma_{-}}$ is invertible while  $P_{\gamma_{+}}$ is surjective Fredholm and the kernel of $P_{\gamma_{+}}$ is identified with $T_{\fa^+}\fU^+ = T_{[\fa^+]}[\fC^+]$. Hence, an orientation of $\det(P_{\gamma})$ is induced by a choice of an element of $\Lambda([\fa^-],[\fa^+])$ along with an orientation of $T_{[\fa^+]}[\fC^+]$.

When the moduli space $\ucM([\fC^-],[\fC^+])$ is not boundary-obstructed, the operator $P_{\gamma}$ is surjective and its kernel is $T_{[\gamma]}\ucM([\fU^-],[\fU^+])$, so $\det T_{[\gamma]}\ucM([\fU^-],[\fU^+]) = \det(P_{\gamma})$. In the boundary-obstructed case, $P_{\gamma}$ picks up a 1-dimensional cokernel and $\det(P_{\gamma}) = \det T_{[\gamma]}\ucM([\fU^-],[\fU^+]) \otimes \operatorname{coker} (P_{\gamma_0})^*$. The line $\operatorname{coker} (P_{\gamma_0})$ has a canonical orientation. In both cases, there is a correspondence between orientations of $T_{[\gamma]}\ucM([\fU^-],[\fU^+])$ and orientations of $\det(P_\gamma)$. Following \cite[Section~20.6]{kronheimer-mrowka}, there is an analogous correspondence for the orientations of reducible moduli spaces. We use $\bar\Lambda([\fa^-],[\fa^+])$ and $\bar\Lambda([\fa])$ as notation for the reducible orientation sets.

\begin{proposition}\label{prop:moduli-spaces-orientable}
    For a regular $\Pin(2)$-non-degenerate perturbation, the moduli spaces are orientable. Moreover, there is a canonical correspondence between orientations of $\ucM([\fC^-],[\fC^+])$ and elements of $\Lambda([\fa^-],[\fa^+])$, for any choice of points $[\fa^{\pm}]\in \fC^{\pm}$. Likewise, there is a canonical correspondence between orientations of $\ucM^{red}([\fC^-],[\fC^+])$ and elements $\bar\Lambda([\fa^-],[\fa^+])$.
\end{proposition} \begin{proof}
    For a $\Pin(2)$-non-degenerate perturbation, the critical submanifolds are all canonically oriented, since every critical submanifold is either an isolated point or a $\CP^1$ with a complex structure. Therefore, for any given solution $[\gamma]\in\ucM([\fC^-],[\fC^+])$ with limit points $[\fa^-]$ and $[\fa^+]$, there is a correspondence between orientations of $T_{[\gamma]}\ucM([\fC^-],[\fC^+])$ and elements of $\Lambda([\fa^-],[\fa^+])$.

    As $[\fa^+]$ is varied within $[\fC^+]$, the orientation sets $\Lambda([\fa^-],[\fa^+])$ form a double cover of $[\fC^+]$. Since $[\fC^+]$ is simply-connected, this double cover is trivial. A choice of section induces an orientation on $\ucM([\fC^-],[\fC^+])$.

    Orientations of the reducible moduli spaces work similarly.
\end{proof} \begin{remark}
    In the proof of Proposition \ref{prop:moduli-spaces-orientable}, we observed that the orientation sets $\Lambda([\fa^-],[\fa^+])$ are identified for different $[\fa^+]$ in the same critical submanifold $[\fC^+]$, or for different $[\fa^-]$ in the same $[\fC^-]$. For the same reason, there is an identification of the orientation sets $\Lambda([\fa])$ for different $[\fa]$ in the same critical submanifold $[\fC]$. When it makes sense, we will write $\Lambda([\fC^-],[\fC^+])$ and $\Lambda([\fC])$.
\end{remark}

In order to define Floer homology, we also need to understand how the orientations on the moduli spaces behave under gluing. This is the same as in the Morse setting. Recall that the orientation sets come with a composition law \begin{equation}\label{eq:orientation-composition-law}
    \Lambda([\fa],[\fb])\times \Lambda([\fb],[\fc]) \to \Lambda([\fa],[\fc])
\end{equation} and a there is a similar composition law for the reducible orientation sets. The codimension-1 strata of a moduli space $\ucM^{+}([\fC^-],[\fC^+])$ are fiber products of intermediate moduli spaces. Using the composition law for orientation sets, we can compose orientations of the factors of such a stratum to get an orientation on $\ucM^{+}([\fC^-],[\fC^+])$, and likewise for reducible moduli spaces. The proofs of \cite[Proposition~20.5.2,~Proposition~20.6.2,~Proposition~20.6.4]{kronheimer-mrowka} can be modified to apply in the $\Pin(2)$-non-degenerate case, telling us how the product orientations on such strata compares to the boundary orientations.

\subsection{Floer Homology}\label{subsection:floer-homology}{

Suppose that $Y$ is equipped with a $\spin$ structure $\fs_Y$. Fix a perturbation $\fq$ all critical points are non-degenerate or all critical points are $\Pin(2)$-non-degenerate and which is regular in the appropriate sense. The set of critical submanifolds can be partitioned into irreducible, boundary-stable, and boundary-unstable critical points as \begin{equation*}
    \Cr = \Cr^ o\sqcup\Cr^ s\sqcup\Cr^ u.
\end{equation*}
If $[\fC]$ is a critical submanifold, let $\Z_{[\fC]}$ denote the infinite cyclic group $\Lambda([\fC])\times_{\Z/2}\Z$, where $\Z/2$ acts freely on $\Lambda([\fC])$ and by the sign representation on $\Z$. For $ x\in\{ o, s, u\}$, define \begin{equation*}
     C^x = \bigoplus_{[\fC]\in \Cr^x} C^{\delta;\cF}_*([\fC];\Z_{[\fC]})
\end{equation*}
where $\cF$ is the family of $\delta$-chains consisting of moduli spaces of trajectories with the forward-evaluation map $\ev^+$ and $C^{\delta;\cF}_*([\fC];\Z_{[\fC]}) = C^{\delta;\cF}_*([\fC];\Z)\otimes \Z_{[\fC]}$. 

For any two critical submanifolds $[\fC^-],[\fC^+]$, a choice of generators $1^{\pm}\in \Z_{[\fC^\pm]}$ induces an orientation on $\ucM^+([\fC^-],[\fC^+])$ via the composition law \eqref{eq:orientation-composition-law}, and there is a well-defined homomorphism \begin{equation*}
    \epsilon = \epsilon([\fC^-],[\fC^+]) : C^{\delta;\cF}_*([\fC^-];\Z_{[\fC^-]}) \longrightarrow C^{\delta;\cF}_*([\fC^+];\Z_{[\fC^+]})
\end{equation*}
such that \begin{equation*}
    \epsilon(\sigma \otimes 1^-) = (\sigma \times_{[\fC^-]} \ucM^+([\fC^-],[\fC^+]))\otimes 1^+.
\end{equation*}
The homomorphism $\epsilon$ does not depend on the choice of generators, and we will omit the generators from our notation.

For the reducible moduli spaces, there are corresponding homomorphisms $\bar\epsilon$. When $[\fC^-]$ is boundary-unstable and $[\fC^+]$ is boundary-stable, the reducible moduli space $\ucM^{red}([\fC^-],[\fC^+])$ is the boundary of $\ucM([\fC^-],[\fC^+])$, and we define $\bar\epsilon: C^{\delta,\cF}_*([\fC^-];\Z_{[\fC^-]}) \to C^{\delta,\cF}_*([\fC^+];\Z_{[\fC^+]})$ by giving $\ucM^{red}([\fC^-],[\fC^+])$ the boundary orientation and taking fiber products. For the other cases, where the moduli space consists entirely of reducibles, $\bar\epsilon$ is defined by comparison with $\epsilon$ in the same way as \cite[(22.6)]{kronheimer-mrowka}.

Let $w$ be the involution defined on homogeneous $\delta$-chains by $w(\sigma)=(-1)^{\dim\sigma}\sigma$. For $ x, y\in\{ o, s, u\}$, we define a homomorphism $\partial_{ y}^{ x}: C^ x\to  C^ y$ by \begin{equation*}
    \partial_{ y}^{ x} = \begin{dcases}
        \sum_{[\fC^-]\in\Cr^ x}\sum_{[\fC^+]\in\Cr^ y}\epsilon([\fC^-],[\fC^+])\circ w,&\qquad  x\neq  y \\
        -\,\partial +\hspace{-10pt} \sum_{[\fC^-]\in\Cr^ x}\sum_{[\fC^+]\in\Cr^ y}\epsilon([\fC^-],[\fC^+])\circ w,&\qquad  x= y
    \end{dcases}
\end{equation*}
where $\partial$ is the internal boundary operator on $\delta$-chains. Likewise, for $ x, y\in\{ s, u\}$, we define a homomorphism $\dbar_{ y}^{ x}$ by \begin{equation*}
    \dbar_{ y}^{ x} = \begin{dcases}
        \sum_{[\fC^-]\in\Cr^ x}\sum_{[\fC^+]\in\Cr^ y}\bar\epsilon([\fC^-],[\fC^+])\circ w,&\qquad  x\neq  y \\
        -\,\partial +  \hspace{-10pt} \sum_{[\fC^-]\in\Cr^ x}\sum_{[\fC^+]\in\Cr^ y}\bar\epsilon([\fC^-],[\fC^+])\circ w,&\qquad  x= y.
    \end{dcases}
\end{equation*}
\begin{definition}\label{def:monopole-chain-complexes}
    The three monopole chain complexes are the chain complexes with underlying abelian groups \begin{equation*}
            \CMbar =  C^ s\oplus  C^ u,\quad \CMto =  C^ o\oplus  C^ s,\quad \CMfrom =  C^ o\oplus C^ u
        \end{equation*}
       and with differentials \begin{equation*}
            \dbar = \begin{bmatrix}
                \dbar_{ s}^{ s} & \dbar_{ s}^{ u} \\
                \dbar_{ u}^{ s} & \dbar_{ u}^{ u}
            \end{bmatrix}, \quad \
            \dcheck = \begin{bmatrix}
                \partial_{ o}^{ o} & -\partial_{ o}^{ u}w\,\dbar_{ u}^{ s} \\
                \partial_{ s}^{ o} & \dbar_{ s}^{ s} - \partial_{ s}^{ u}w\,\dbar_{ u}^{ s}
            \end{bmatrix},\quad \
            \dhat = \begin{bmatrix}
                \partial_{ o}^{ o} & \partial_{ o}^{ u} \\
                -\dbar_{ u}^{ s}w\,\partial_{ s}^{ o} & -\dbar_{ u}^{ u} - \partial_{ u}^{ s}w\,\partial_{ s}^{ u}
            \end{bmatrix}.
    \end{equation*}
\end{definition}
\begin{proposition}\label{prop:differentials}
    The differentials in Definition \ref{def:monopole-chain-complexes} satisfy \begin{equation*}
        \dbar^2 = \dcheck^2=\dhat^2=0.
    \end{equation*}
\end{proposition} \begin{proof}
    We will prove the identity \begin{equation*}
        - \partial_{ o}^{ o}\partial_{ o}^{ o}\
        + \partial_{ o}^{ u}w\dbar_{ u}^{ s}\partial_{ s}^{ o} = 0.
    \end{equation*}
    This will verify that the top-left matrix entry of $\dcheck^2$ vanishes. The proof is similar for the other terms. Consider a chain $\sigma\in C_*^\delta([\fC^-];\Z_{[\fC^-]})$. We have \begin{equation}\label{eq:do2}
        \begin{split}
            \partial_ o\partial_ o \sigma &= \partial_{ o}\left(-\partial \sigma + (-1)^{\dim\sigma}\sum_{[\fC^+]\in\Cr^ o} \sigma \times_{[\fC^-]} \ucM^+([\fC^-],[\fC^+])\right) \\
            &= \partial^2\sigma - \sum_{[\fC^+]\in\Cr^ o}(-1)^{\dim\sigma}\partial (\sigma \times_{[\fC^-]} \ucM^+([\fC^-],[\fC^+])) \\
            & + \sum_{[\fA]\in\Cr^{ o}} \sum_{[\fC^+]\in\Cr^{ o}}(-1)^{\dim\sigma + \dim(\sigma \times_{[\fC^-]} \ucM^+([\fC^-],[\fA]))}\sigma \times_{[\fC^-]}\ucM^+([\fC^-],[\fA],[\fC^+]) \\
            &= \partial^2\sigma - \sum_{[\fC^+]\in\Cr^ o}(-1)^{\dim\sigma} \partial (\sigma \times_{[\fC^-]} \ucM^+([\fC^-],[\fC^+])) \\
            & + \sum_{[\fA]\in\Cr^{ o}} \sum_{[\fC^+]\in\Cr^{ o}} (-1)^{\dim\ucM^+([\fC^-],[\fA])}\sigma \times_{[\fC^-]} \ucM^+([\fC^-],[\fA],[\fC^+]) \\
            &= - \sum_{[\fC^+]\in\Cr^ o}(-1)^{\dim\sigma}\partial (\sigma \times_{[\fC^-]} \ucM^+([\fC^-],[\fC^+])) \\
            & + \sum_{[\fA]\in\Cr^{ o}} \sum_{[\fC^+]\in\Cr^{ o}}(-1)^{\dim\ucM^+([\fC^-],[\fA])}\sigma \times_{[\fC^-]} \ucM^+([\fC^-],[\fA],[\fC^+])
        \end{split}
    \end{equation}
    where in the second to last equation of \eqref{eq:do2} we use $\dim(\sigma \times_{[\fC^-]} \ucM^+([\fC^-],[\fA])) = \dim\sigma + \dim \ucM^+([\fC^-],[\fA]) - \dim[\fC^-]$ and the fact that $\dim[\fC^-]$ is even. By a similar expansion, we have \begin{equation}\label{eq:dobs}
        \partial_{ o}^{ u}w\dbar_{ u}^{ s}\partial_{ s}^{ o} = \sum_{[\fA]\in\Cr^{ s}}\sum_{[\fB]\in\Cr^{ u}}\sum_{[\fC^+]\in\Cr^{ o}} (-1)^{\dim\ucM^+([\fC^-],[\fA])}\sigma \times_{[\fC^-]} \ucM^+([\fC^-],[\fA],[\fB],[\fC^+]).
    \end{equation}
    Combining \eqref{eq:do2} and \eqref{eq:dobs} yields
    \begin{equation}\label{eq:d22}
        \begin{split}
            & (-1)^{\dim\sigma}\left(- \partial_{ o}^{ o}\partial_{ o}^{ o}\
        + \partial_{ o}^{ u}w\dbar_{ u}^{ s}\partial_{ s}^{ o}\right)\sigma = \sum_{[\fC^+]\in\Cr^ o} \partial(\sigma \times_{[\fC^-]} \ucM^+([\fC^-],[\fC^+])) \\
        & - \sum_{[\fC^+]\in\Cr^ o} (\partial\sigma) \times_{[\fC^-]} \ucM^+([\fC^-],[\fC^+]) \\
        & - (-1)^{\dim\sigma}\sum_{[\fA]\in\Cr^{ o}}\sum_{[\fC^+]\in\Cr^{ o}} (-1)^{\dim \ucM^+([\fC^-],[\fA])}  \sigma\times_{[\fC^-]}\ucM^+([\fC^-],[\fA],[\fC^+]) \\
        & - (-1)^{\dim\sigma} \sum_{[\fA]\in\Cr^{ s}}\sum_{[\fB]\in\Cr^{ u}}\sum_{[\fC^+]\in\Cr^{ o}}(-1)^{\dim\ucM^+([\fC^-],[\fA]) + 1}\sigma \times_{[\fC^-]} \ucM^+([\fC^-],[\fA],[\fB],[\fC^+]).
        \end{split}
    \end{equation}
    A version of \cite[Proposition~20.5.2]{kronheimer-mrowka} holds in our setting, so that the boundary of $\ucM^+([\fC^-],[\fC^+])$ consists of faces of the form $(-1)^{\dim \ucM^+([\fC^-],[\fA])}\ucM^+([\fC^-],[\fA],[\fC^+])$ and boundary-obstructed faces of the form $(-1)^{\dim\ucM^+([\fC^-],[\fA])+1}\ucM^+([\fC^-],[\fA],[\fB],[\fC^+])$. By Lemma \ref{lem:fiber-product-boundary-orientation}, the right-hand side of \eqref{eq:d22} vanishes.
\end{proof} \begin{remark}
    Notice how the placement of the involution $w$ in the definition of the differential causes the necessary signs to emerge. The reader might compare to the case of finite-dimensional Morse-Bott theory, \cite[Lemma~3.4]{austin-braam}.
\end{remark}

The monopole Floer complexes have relative $\Z$-gradings and absolute gradings by the transitive $\Z$-set $\bJ(\fs_Y)$, as in \cite[ch.~3,~Section~2]{f-lin}. There is a canonical $\Z$-equivariant map $\bJ\to \Z/2$, and we write $\gr^{(2)}$ for the $\Z/2$ grading. In the case that $c_1(\fs_Y)$ is torsion, the action of $\Z$ on $\bJ$ is free. We write $\gr$ for the relative $\Z$-grading and $\Gr$ for the absolute $\bJ$-grading. The $\bJ$-grading along with the $\Z$-grading coming from the dimension of $\delta$-chains make the complexes $\CMbar,\CMto,\CMfrom$ into bi-graded complexes such that the differentials have total degree $-1$.

\begin{definition}
    The monopole Floer homology groups are the $\bJ$-graded abelian groups \begin{equation*}
        \HMbar_*(Y,\fs_Y)=H_*(\CMbar),\quad\HMto_*(Y,\fs_Y)=H_*(\CMto),\quad\HMfrom_*(Y,\fs_Y)=H_*(\CMfrom).
    \end{equation*}
\end{definition}

\begin{remark}
    If $[\fC]=[\fa]$ is a genuinely non-degenerate critical point, so that $\dim[\fC]=0$, then \begin{equation*}
        C_*^{\delta;\cF}([\fC];\Z_{[\fC]})\cong\Z_{[\fC]}. 
    \end{equation*}
    This follows from Lemma \ref{lem:degenerate-1-chains}, which implies that there all positive-dimensional $\delta$-chains in a 0-dimensional manifold are degenerate. When the perturbation used is genuinely non-degenerate, our monopole Floer homology groups are the same as those in \cite{kronheimer-mrowka}. 
\end{remark}

There is a long exact sequence relating the three versions of monopole Floer homology, \begin{equation}\label{eq:long-exact-sequence}
    \begin{tikzcd}
        \cdots \ar[r] & \HMbar(Y,\fs_Y) \ar[r,"i_*"] & \HMto_*(Y,\fs_Y) \ar[r,"j_*"] & \HMfrom_*(Y,\fs_Y) \ar[r,"p_*"] & \HMbar_{*-1}(Y,\fs_Y) \ar[r] & \cdots
    \end{tikzcd}
\end{equation} 
where the maps $i_*,j_*,p_*$ are induced by (anti-)chain maps \begin{equation*}
    i =\begin{bmatrix}
        0 & -\partial_{ o}^{ u} \\
        1 & -\partial_{ s}^{ u}
    \end{bmatrix},\quad j= \begin{bmatrix}
        1 & 0 \\
        0 & -\dbar_{ u}^{ s}
    \end{bmatrix}, \quad p = \begin{bmatrix}
        \partial_{ s}^{ o} & \partial_{ s}^{ u} \\
        0 & 1
    \end{bmatrix}.
\end{equation*}
\begin{definition}
    The reduced monopole Floer homology group is \begin{equation*}
        \HMred(Y,\fs_Y) = \im j_* = \ker p_* = \coker i_*.
    \end{equation*}
\end{definition}
In the case of a $\Pin(2)$-non-degenerate perturbation, there is an involution $\jmath$ on each of the three complexes $\CMbar,\CMto,\CMfrom$ induced by the action of $j\in\Pin(2)$. \begin{proposition}\label{prop:j-on-homology}
    The involution $\jmath$ induces involutions $\jmath_*$ on $\HMbar,\HMto,\HMfrom$ with respect to which the long exact sequence \eqref{eq:long-exact-sequence} is natural.
\end{proposition} \begin{proof}
    The involution $\jmath$ induces isomorphisms $\Z_{[\fC]}\xrightarrow{\sim} \Z_{\jmath[\fC]}$ and $\ucM^+([\fC^-],[\fC^+])\xrightarrow{\sim}\ucM^+(\jmath[\fC^-],\jmath[\fC^+])$, affecting the orientations such that $\jmath$ commutes with the isomorphisms $\epsilon$ and $\bar\epsilon$. Since $\jmath$ is homogeneous, it also commutes with  $w$. It follows that $\jmath$ commutes with each of $\dbar,\dcheck,\dhat,i,j,p$.
\end{proof} \begin{remark}
    In the proof of Proposition \ref{prop:j-on-homology}, we implicitly use the identification $\bar\fs_Y=\fs_Y$ to identify the coefficient groups $\Z_{[\fC]}$ and $\Z_{\jmath[\fC]}$. In general, these groups are only identified after tensoring with a system of local coefficients; see \cite[\textsection~25.5]{kronheimer-mrowka}.
\end{remark}
From Proposition \ref{prop:j-on-homology}, it follows that $\jmath$ also induces an involution $\jmath_*$ on $\HMred$.

}

\subsection{Cobordisms \& Invariance}\label{subsection:cobordisms-and-invariance}{

In Subsection \ref{subsection:floer-homology}, we defined the monopole Floer homology groups associated to $Y$. The definition requires the choice of a Riemannian metric and a perturbation. Following \cite{bloom} and \cite{f-lin}, we will now proceed to define maps on the (completed) monopole Floer homology groups induced by cobordisms, and use these to prove that the monopole Floer homology groups do not depend on the metric or perturbation chosen. 

Let $(Y^{-},\fs_{Y^{-}}),(Y^{+},\fs_{Y^+})$ be $\spin$ 3-manifolds and $(W,\fs_W):(Y^{-},\fs_{Y^{-}})\to (Y^{+},\fs_{Y^+})$ be a $\spin$ cobordism. The direction of the arrow indicates that the oriented boundary of $W$ is $\partial W = -Y^{-}\sqcup Y^{+}$. Let $p_1,\ldots,p_m \in W$ be points in the interior and let $W_{\vec p}$ be the cobordism $Y^{-}\sqcup S^3\sqcup\ldots\sqcup S^3\to Y^{+}$ obtained from $W$ by removing small open balls around each of the points $p_i$. Let $S_i^3$ denote the $S^3$ component corresponding to the point $p_i$. Let $W_{\vec p, \infty}$ be the open 4-manifold obtained from $W_{\vec p}$ by attaching cylindrical ends to each of the boundary components. 

Fix perturbations over each of the boundary components of $W_{\vec p}$ which $\Pin(2)$-non-degenerate and regular in the appropriate sense, and such that over each $S_i^3$ there is a unique (reducible) critical point and no irreducible critical points before blowing up. Extend these perturbations to a regular perturbation over $W_{\vec p,\infty}$. We can use the moduli spaces on $W_{\vec p, \infty}$ to define maps \begin{equation*}
    HM^{\circ}(W; \vec p): HM^{\circ}(Y^{-},\fs_{Y^-}) \otimes \HMfrom(S^3)\otimes \cdots \HMfrom(S^3) \to HM^{\circ}(Y^{+},\fs_{Y^+})
\end{equation*}
where $HM^{\circ} \in \{\HMbar,\HMto,\HMfrom\}$. Orienting the relevant moduli spaces as $\delta$-chains requires a few words. \begin{example}\label{ex:cobordism}
    Let $[\fC^{-}],[\fC^{+}]$ be Morse-Bott critical submanifolds over $-Y^{-},Y^{+}$, respectively and let $[\fC_i]$ be boundary-stable critical submanifolds over $-S_i^3$ for $i\in \{1,\ldots,m\}$. Consider the Seiberg--Witten moduli space \begin{equation*}
        \Delta = \cM^+(W_{\vec p, \infty}; [\fC^{-}], [\fC_1], \ldots, [\fC_m], [\fC^{+}]).
    \end{equation*}
    The codimension-1 faces of such a moduli space are classified according to \cite[Proposition~24.6.10]{kronheimer-mrowka}. There are \begin{enumerate}[label=(\roman*)]
        \item moduli spaces broken at one of the ends, \begin{equation*}
            \Sigma = \cM^+(W_{\vec p, \infty}; [\fA], [\fC_1], \ldots, [\fC_m], [\fC^{+}]) \times \ucM^+([\fA],[\fC^-])
        \end{equation*} or \begin{equation*}
            \Sigma = \cM^+(W_{\vec p, \infty}; [\fC^{-}], [\fC_1], \ldots, \hspace{-9pt}\underset{\text{repl. }[\fC_i]}{[\fA]}\hspace{-9pt}, \ldots [\fC_m], [\fC^{+}]) \times  \ucM^+([\fA],[\fC_i])
        \end{equation*} or \begin{equation*}
            \Sigma = \cM^+(W_{\vec p, \infty}; [\fC^-], [\fC_1], \ldots, [\fC_m], [\fA]) \times \ucM^+([\fA],[\fC^+]),
        \end{equation*}
        for which $N_{\Delta,\Sigma}=\{[\fA]\}$ and $O_{\Delta,\Sigma}=\emptyset$.
        \item moduli spaces broken at one of the ends with a boundary-obstructed component, \begin{equation*}
            \Sigma = \cM^+(W_{\vec p, \infty}; [\fA], [\fC_1], \ldots, [\fC_m], [\fC^{+}]) \times \ucM^+([\fA],[\fB],[\fC^-])
        \end{equation*} or \begin{equation*}
            \Sigma = \cM^+(W_{\vec p, \infty}; [\fC^{-}], [\fC_1], \ldots, \hspace{-9pt}\underset{\text{repl. }[\fC_i]}{[\fA]}\hspace{-9pt}, \ldots [\fC_m], [\fC^{+}]) \times \ucM^+([\fA],[\fB],[\fC_i])
        \end{equation*} or \begin{equation*}
            \Sigma = \cM^+(W_{\vec p, \infty}; [\fC^-], [\fC_1], \ldots, [\fC_m], [\fC])\times \ucM^+([\fA],[\fB],[\fC^+]),
        \end{equation*}
        for which $N_{\Delta,\Sigma}=\{[\fA],[\fB]\}$ and $O_{\Delta,\Sigma}=\{[\fB]\}$, ordered such that $[\fA]<[\fB]$.
        \item moduli spaces broken at exactly two ends, one of which must be the $Y^+$ end. These are boundary-obstructed of corank 1, \begin{equation*}
            \Sigma = \cM^+(W_{\vec p, \infty}; [\fA], [\fC_1], \ldots, [\fC_m], [\fB]) \times \ucM^+([\fA],[\fC^-])\times \ucM^+([\fB],[\fC^+]), 
        \end{equation*} or \begin{equation*}
            \Sigma = \cM^+(W_{\vec p, \infty}; [\fC^{-}], [\fC_1], \ldots, \hspace{-9pt}\underset{\text{repl. }[\fC_i]}{[\fA]}\hspace{-9pt}, \ldots [\fC_m], [\fB]) \times \ucM^+([\fA],[\fC_i])\times \ucM^+([\fB],[\fC^+]),
        \end{equation*}
        for which $N_{\Delta,\Sigma}=\{[\fA],[\fB]\}$ and $O_{\Delta,\Sigma}=\{[\fB]\}$, ordered such that $[\fA]<[\fB]$.
        \item in the case that the moduli both irreducibles and reducibles, the reducible part \begin{equation*}
            \Sigma = \ucM^{red+}(W_{\vec p, \infty}; [\fC^{-}], [\fC_1], \ldots, [\fC_m], [\fC^{+}]),
        \end{equation*}
        for which $N_{\Delta,\Sigma}=\{[\fC^+]\}$ and $O_{\Delta,\Sigma}=\emptyset$.
    \end{enumerate}
    We can also consider Seiberg--Witten moduli spaces parameterized over some manifold $P$, possibly with boundary, \begin{equation*}
        \Delta = \cM^+(W_{\vec p, \infty}; [\fC^{-}], [\fC_1], \ldots, [\fC_m], [\fC^{+}])_P.
    \end{equation*} 
    Then, there are strata of the four types listed above as well as \begin{itemize}
        \item[(v)] the moduli space of unbroken solutions lying over $\partial P$, \begin{equation*}
            \Sigma = \cM^+(W_{\vec p, \infty}; [\fC^{-}], [\fC_1], \ldots, [\fC_m], [\fC^{+}])_{\partial P},
        \end{equation*}
        for which $N_{\Delta,\Sigma}=\{\partial P\}$ and $O_{\Delta,\Sigma}=\emptyset$.
    \end{itemize}
    In the case that $\Delta$ consists entirely of reducibles, the moduli space has the structure of a manifold-with boundary along each of its codimension-1 strata, so that $N$ consists of a single element and $O=\emptyset$.

    For a general inclusion of faces $\Sigma_a\subseteq\Sigma$, the set $N_{\Sigma,\Sigma_a}$ records breaking of solutions, and $O_{\Sigma,\Sigma_a}$ records boundary-obstructed breaking. When $\Sigma\subseteq \Sigma_a$ has codimension 1, the set $N_{\Sigma,\Sigma_a}$ is ordered as above. The ordering on the set of critical points gives preferred orientations on $(0,\infty]^{N_{\Sigma,\Sigma_{a}}}$ and $\R^{O_{\Sigma,\Sigma_{a}}}$ for every codimension-1 inclusion of faces $\Sigma_a\subseteq\Sigma$.

    We may also consider moduli spaces of solutions parameterized by some oriented base manifold $P$ with boundary. These are oriented similarly.
\end{example}

The moduli spaces considered in Example \ref{ex:cobordism} may also be viewed as moduli spaces \begin{equation*}
    \cM^+( [\fC^{-}], [\fC_1], \ldots, [\fC_m], W_{\vec p, \infty}, [\fC^+])_P
\end{equation*}
where we view $[\fC^-]$ as a critical point over the incoming end $Y^-$ and $[\fC_i]$ as boundary-unstable critical points over the incoming ends $S_i^3$ for $i\in\{1,\ldots,m\}$. Such moduli spaces are oriented relative to the orientation sets $\Lambda([\fC^-]),\Lambda([\fC^+])$, and Example \ref{ex:cobordism} explains how to give these moduli spaces the structure of orientable abstract $\delta$-chains. Note that critical points over the incoming ends $Y^-,S_i^3$ have the opposite order compared to the corresponding critical points over $-Y^-,-S_i^3$. Given a moduli space \begin{equation*}
    \Delta = \cM^+( [\fC^{-}], [\fC_1], \cdots, [\fC_m], W_{\vec p, \infty}, [\fC^+])_P,
\end{equation*}
there is an evaluation map \begin{equation*}
    \ev^- : \Delta \to [\fC^-]\times [\fC_1]\times \cdots \times [\fC_m].
\end{equation*}
Let $\cF_W$ be the family of $\delta$-chains given by such evaluation maps.

As in \cite[ch.~3,~Section~3]{f-lin}, for $ x\in \{ o, s, u\}$, consider the subgroup \begin{equation*}
     C^x(Y^-,\fs_{Y^-}; \vec p) \subseteq  C^ x(Y^-,\fs_{Y^-})\otimes  C^ u (S_1^3) \otimes \cdots \otimes  C^ u(S_m^3)
\end{equation*}
generated by products $\sigma^-\otimes \sigma_1\otimes\ldots\otimes \sigma_m$ such that $\sigma^-\times \sigma_1\times\ldots\times \sigma_m$ is transverse to all $\delta$-chains in the family $\cF_W$. Define \begin{equation*}
    \begin{split}
        \CMbar(Y^-,\fs_{Y^-};\vec p) &=  C^ s(Y^-,\fs_{Y^-};\vec p) \oplus  C^ u(Y^-, \fs_{Y^-};\vec p) \\
        \CMto(Y^-,\fs_{Y^-};\vec p) &=  C^ o(Y^-,\fs_{Y^-};\vec p) \oplus  C^ s(Y^-, \fs_{Y^-};\vec p)  \\
        \CMfrom(Y^-,\fs_{Y^-};\vec p) &=  C^ o(Y^-,\fs_{Y^-};\vec p) \oplus  C^ u(Y^-, \fs_{Y^-};\vec p).
    \end{split}
\end{equation*}
We chose perturbations over the ends $S_i^3$ such that $\CMfrom(S_i^3) =  C^ u(S_i^3)$ as abelian groups. For each flavor of monopole Floer complex $CM^{\circ}\in\{\CMbar,\CMto,\CMfrom\}$, the subgroup \begin{equation*}
    CM^{\circ}(Y^-, \fs_{Y^-};\vec p)\subseteq CM^{\circ}(Y^-, \fs_{Y^-})\otimes\bigotimes_{i=1}^m \CMfrom(S_i^3)
\end{equation*}
is a subcomplex and the inclusion is a quasi-isomorphism; see \cite[ch.~3,~Lemma~3.1]{f-lin} for the characteristic 2 case. The homology $\HMfrom(S^3)$ is $\Z[U]$, where $\deg U = -2$. The homology of $CM^\circ(Y^-; \vec p)$ is canonically identified with \begin{equation*}
    HM^\circ(Y^-, \fs_{Y^-}; \vec p) = CM^\circ(Y^-, \fs_{Y^-}) \otimes \bigotimes_{i=1}^{m} \Z[U_i].
\end{equation*}

Suppose that $W$ is given a homology orientation. Analogously to the maps $\epsilon$ defined in the previous subsection, given critical submanifolds $[\fC^\pm]\in \Cr(Y^\pm), [\fC_i]\in \Cr^u(S_i^3)$, we can define \begin{equation*}
    \begin{split}
            \mu = \mu([\fC^-],[\fC_1],\ldots,[\fC_m],[\fC^+]): 
            C_*^{\delta;\cF_W}([\fC^-];\Z_{[\fC^-]}) \otimes \bigotimes_{i=1}^{m} C_*^{\delta;\cF_W}([\fC_i];\Z_{[\fC_i]}) \to C_*^{\delta;\cF}([\fC^+];\Z_{[\fC^+]})
    \end{split}
\end{equation*}
by taking fiber products with the moduli space $\cM_{\fs_W}^+([\fC^-],[\fC_i],\ldots,[\fC_m],W_{\vec p,\infty},[\fC^+])$, in the same way that $\epsilon$ was defined. Likewise, we define maps $\bar\mu$ by taking fiber products with reducible moduli spaces. As with $\epsilon$, in the unstable-stable case, the reducible moduli space is oriented as the boundary of the full moduli space, and in the unstable-unstable case, $\bar\mu$ and $\mu$ differ by a possible sign; see the remarks after \cite[(25.6)]{kronheimer-mrowka}.

For $x,y\in\{o,s,u\}$, we define homomorphisms $m_y^x:C^x(Y^-,\fs_{Y^-};\vec p)\to C^y(Y^+,\fs_{Y^+})$ by \begin{equation*}
    m_y^x = m_y^x(W,\fs_W; \vec p) = \sum_{[\fC^-]\in \Cr^x(Y^-)} \sum_{\underset{i\in\{1,\ldots,m\}}{[\fC_i]\in\Cr^u(S_i^3)}}\sum_{[\fC^+]\in \Cr^y(Y^+)} \, \mu([\fC^-],[\fC_1],\ldots,[\fC_m],[\fC^+]) \circ w.
\end{equation*}
Likewise, for $x,y\in\{s,u\}$, we define homomorphisms $\bar m_y^x: C^x(Y^-, \fs_{Y^-};\vec p)\to C^y(Y^+)$ by \begin{equation*}
    \bar m_y^x = \bar m_y^x(W,\fs_W; \vec p) = \sum_{[\fC^-]\in \Cr^x(Y^-)} \sum_{\underset{i\in\{1,\ldots,m\}}{[\fC_i]\in\Cr^u(S_i^3)}}\sum_{[\fC^+]\in \Cr^y(Y^+)} \bar\mu([\fC^-],[\fC_1],\ldots,[\fC_m],[\fC^+]) \circ w.
\end{equation*}
We define maps on the chain complexes \begin{equation*}
    \begin{split}
        \mbar = \mbar(W,\fs_W; \vec p) &: \CMbar(Y^-,\fs_{Y^-};\vec p) \longrightarrow \CMbar(Y^+,\fs_{Y^+}) \\
        \mcheck = \mcheck(W,\fs_W; \vec p) &: \CMto(Y^-,\fs_{Y^-};\vec p) \longrightarrow \CMto(Y^+,\fs_{Y^+})  \\
        \mhat = \mhat(W,\fs_W; \vec p) &: \CMfrom(Y^-,\fs_{Y^-};\vec p) \longrightarrow \CMfrom(Y^+,\fs_{Y^+}) 
    \end{split}
\end{equation*}
by $\mbar = \begin{bmatrix}
        \mbar_s^s & \mbar_s^u \\
        \mbar_u^s & \mbar_u^u
    \end{bmatrix}$, $\mcheck = \begin{bmatrix}
        m_o^o & -m_o^u w\, \dbar_u^s(Y^-;\vec p) - \partial_o^u(Y^+) w \, \mbar_u^s \\
        m_s^o& \mbar_s^s - \partial_s^u(Y^+) w\, \mbar_u^s
    \end{bmatrix}$, and \begin{equation*}
     \mhat = \begin{bmatrix}
        m_o^o & m_o^u \\
        \mbar_u^s w\, \partial_{s}^{o}(Y^-;\vec p) - \dbar_{u}^{s}(Y^+)m_{s}^{o} & \mbar_{u}^{u} + \mbar_{u}^{s} w\, \partial_{s}^{u}(Y^-; \vec p) - \dbar_{u}^{s}(Y^+) w \, m_{s}^{u}
    \end{bmatrix}.
\end{equation*} \begin{remark}
    Whereas the analogous maps in \cite{f-lin} were defined only on the negative $\bJ$-completions, we do not need to work with completions here as we have fixed a $\spin$ (and hence $\spinc$) structure on $W$.
\end{remark} 
\begin{proposition}\label{prop:chain-maps}
    The maps $\mbar,\mcheck,\mhat$ are chain maps. At the level of homology, the maps \begin{equation*}
        m^{\circ} : HM^{\circ}(Y^-,\fs_{Y^-}) \otimes \Z[U_1] \otimes \cdots \otimes \Z[U_m] \longrightarrow HM^{\circ}(Y^+,\fs_{Y^+}) 
    \end{equation*}
    do not depend on the placement of the punctures $\vec p$ nor on the metric nor perturbation on $W$ (assuming a fixed cylindrical metric and perturbation near the boundary).
\end{proposition} \begin{proposition}\label{prop:u-maps}
    The induced maps \begin{equation*}
        m_*^{\circ} : HM^{\circ}(Y^-,\fs_{Y^-}) \otimes \Z[U_1] \otimes \cdots \otimes \Z[U_m] \longrightarrow HM^{\circ}(Y^+,\fs_{Y^+}) 
    \end{equation*}
    factor uniquely through the change-of-coefficients homomorphism \begin{equation*}
            HM^{\circ}(Y^-,\fs_{Y^-}) \otimes \Z[U_1] \otimes \cdots \otimes \Z[U_m] \longrightarrow  HM^{\circ}(Y^-,\fs_{Y^-}) \otimes \Z[U]
    \end{equation*}
    induced by \begin{equation*}
        U_1^{d_1} \otimes \cdots\otimes U_m^{d_m} \longmapsto U^{d_1+\cdots+d_m}.
    \end{equation*}
\end{proposition}
\begin{proof}[Proofs of Propositions \ref{prop:chain-maps}, \ref{prop:u-maps}]
    The propositions follow by adding signs to the proofs in the characteristic 2 case; see \cite[ch.~3,~Proposition~3.2 and Lemma~3.4]{f-lin}. The appropriate signs emerge due to the placement of the involution $w$, as in the proof of Proposition \ref{prop:differentials}, which one may in turn compare to the proof of \cite[Proposition~2.4]{f-lin}.
\end{proof} \begin{remark}
    The proofs of Propositions \ref{prop:chain-maps} and \ref{prop:u-maps} use chain homotopies defined using moduli spaces parameterized over an interval. These can be made to allow for both $\Pin(2)$-non-degenerate and genuinely non-degenerate perturbations. In particular, the induced maps at the level of homology are the same as in \cite{kronheimer-mrowka}.
\end{remark}

Combining Propositions \ref{prop:chain-maps} and \ref{prop:u-maps}, for a fixed metric and perturbation on $Y$, we have well-defined maps \begin{equation*}
    HM^\circ(\, U^d \,|\, W, \fs_{W} \,) = m^\circ(-\otimes U^{d_1}\otimes \cdots \otimes U^{d_m})_* : HM^\circ(Y^-,\fs_{Y^-}) \to HM^\circ(Y^+,\fs_{Y^+}).
\end{equation*}
independent of the metric and perturbation on $W$ and of the number $m>0$ and placement of the punctures. We will be most interested in the case $d=0$. We write $W_*^\circ$ in place of $HM^\circ(\,1\,|\,W,\fs_W\,)$. In particular, we write $W_*$ for the induced map on $HM(Y,\fs_Y)$. These maps are functorial with respect to composition of cobordisms. \begin{proposition}
    If $(W,\fs_W)=I\times (Y,\fs_Y,g,\fq)$ is the trivial cobordism, then $W^\circ$ is the identity map. If $W = W_1\circ W_2$ is a composite cobordism, then \begin{equation*}
        HM^{\circ}(\, U^{d_1+d_2} \,|\, W,\fs_W \,) = HM^{\circ}(\, U^{d_1} \,|\, W_1,\fs_{W_1} \,)\circ HM^{\circ}(\, U^{d_2} \,|\, W_2,\fs_{W_2} \,)
    \end{equation*}
    where $\fs_{W_1},\fs_{W_2}$ are the restricted $\spin$ structures.
\end{proposition}\begin{proof}
    The proposition follows from adding signs to the proofs in the characteristic 2 case; see \cite[Propositions~3.6 and 3.6]{f-lin}.
\end{proof}

\begin{remark}
    While we needed to introduce punctures in order to define the maps $HM^\circ(\, U^d \,|\, W,\fs_W \,)$ for $d>0$, we could have defined $W*^\circ$ without introducing any punctures. We will make use of this observation later.
\end{remark}

Using the fact that the cobordism maps are independent of the metric and perturbation on the cobordism, we can finally prove that the monopole Floer homology groups are invariants of $(Y,\fs_Y)$, independent of the metric and perturbation on $Y$. \begin{proposition}
    The monopole Floer homology groups $\HMbar(Y,\fs_Y),\HMto(Y,\fs_Y),\HMfrom(Y,\fs_Y)$ for different choices of metric and perturbation are canonically isomorphic. When $\fs_Y$ is induced by a $\spin$ structure, the involution $\jmath$ does not depend on the choice of regular $\Pin(2)$-non-degenerate perturbation.
\end{proposition} \begin{proof}
    The claim follows from the same argument used in the characteristic 2 case; see \cite[Corollary~3.7]{f-lin}. The main point is that a path interpolating between two different choices of metric and pertubration on $Y$ gives a metric and perturbation on the cylinder, and the induced map gives a canonical isomorphism between the two monopole Floer homology groups. When the perturbations are both $\Pin(2)$-non-degenerate, we can choose a path interpolating through $j$-equivariant perturbations, so that the canonical isomorphism is $\jmath$-equivariant.
\end{proof}

}


\section{Solution Counts and Lefschetz Numbers}\label{section:solution-counts}

\noindent In this section, we return to the situation of Theorem \ref{thm:deg-formula}. Throughout, $\tildeX$ is the double-cover of a homology $S^1\times S^3$ with real $\spinc$ structure $\fs_{\tildeX}$. The real involution $I$ is defined by combining the action of $j$ coming from the even-type $\spin$ structure with the gauge transformation $u$, as in Subsection \ref{subsection:real-spin-structures}. The hypersurface $Y$ is a rational homology sphere.  Both $Y$ and $W$ are given the $\spin$ structures restricted from the even-type $\spin$ structure on $\tildeX$. These $\spin$ structures induce $\spinc$ structures $\fs_Y$ and $\fs_W$, respectively. 

\subsection{Perturbations on Rational Homology Spheres}\label{subsection:perturbations-on-rational-homology-spheres}{

Since we are working with a rational homology sphere, we can restrict the class of perturbations that we use, so as to simplify the monopole Floer complex. 

A perturbation $\fq$ over $Y$ is \emph{nice} if it is regular in the appropriate sense and vanishes identically along the reducible locus. The proof of \cite[Proposition~2.8]{j-lin} can be adapted to our setting to establish the existence of nice $\Pin(2)$-non-degenerate perturbations. We will assume such a perturbation has been fixed. For a nice perturbation, dimensional constraints force the operators $\partial_y^x$ to vanish for $x,y\in\{u,s\}$. Therefore, the differentials in Definition \ref{def:monopole-chain-complexes} take the simplified forms \begin{equation*}
    \dbar = \begin{bmatrix}
        -\partial & 0 \\
        0 & -\partial 
    \end{bmatrix},\quad \dcheck = \begin{bmatrix}
        \partial_{  o}^{  o} & 0 \\
        \partial_{  s}^{  o} & -\partial
    \end{bmatrix},\quad \dhat = \begin{bmatrix}
        \partial_{  o}^{  o} & \partial_{  o}^{  u} \\
        0 & -\partial
    \end{bmatrix}
\end{equation*}
where $\partial$ is the internal boundary operator on $\delta$-chains. The (anti-)chain maps $i,j,p$ also simplify to \begin{equation*}
    i = \begin{bmatrix}
        0 & -\partial_{  o}^{  u} \\
        1 & 0
    \end{bmatrix},\quad j = \begin{bmatrix}
        1 & 0 \\
        0 & 0
    \end{bmatrix},\quad p = \begin{bmatrix}
        \partial_{  s}^{  o} & 0 \\
        0 & 1
    \end{bmatrix}.
\end{equation*}
For a nice perturbation, the reducibles upstairs belong to a single tower of $\CP^1$ families lying over the same reducible downstairs.

The Morse-Bott comples associated to a genuinely non-degenerate perturbation is the same as the ordinary monopole Floer complex. The Morse-Bott complex associated to a $\Pin(2)$-non-degenerate perturbation is more unwieldly due to its being non-finitely-generated in each grading. To deal with this issue, we would like to pass from $\delta$-chains to the homology of the critical submanifolds. The simplified form of the complex in the presence of a nice perturbation gives us one way to achieve locally-finitely-generatedness. \begin{proposition}\label{prop:grading-filtration}
    Suppose $Y$ is a rational homology 3-sphere, $\fs_Y$ is induced by a $spin$ structure on $Y$, and $\fq$ is a nice regular $\Pin(2)$ non-degenerate perturbation. Then, there exists a $\Z,\bJ$-bigraded spectral sequence $\widecheck{E}$ converging to $\HMto(Y,\fs_Y)$ such that the $E^0$ page is \begin{equation*}
        \widecheck E^0 = \CMto(Y,\fs_Y;\fq)
    \end{equation*}
    and the $E^1$ page is \begin{equation*}
        \widecheck E_{ij}^1 = \bigoplus_{\underset{\operatorname{Gr}([\fC]) = j}{[\fC]\in \Cr^o\sqcup\Cr^s}} H_i([\fC];\Z_{[\fC]}).
    \end{equation*}
\end{proposition}\begin{proof}
    Since we are working with a nice perturbation, contributions to the differential all come from moduli spaces of unbroken trajectories, and there are no contributions coming from trajectories between reducibles. We claim that if the moduli space $\ucM_z^+([\fC^-],[\fC^+])$ contributes to the differential $\dcheck$, then $\gr_z([\fC^-],[\fC^+]) \geq 1.$ Once we prove this claim, the lemma will follow according the same argument used to prove \cite[ch.~3,~Proposition~2.9]{f-lin}.
    
    Since we are working with the ``to" version of the complex and there are no contributions from trajectories between reducibles, $[\fC^-]$ is irreducible. The relative grading is related to the dimension of the moduli space by \begin{equation*}
        \gr_z([\fC^-],[\fC^+]) = 1 + \dim\ucM_z^+([\fC^-],[\fC^+]) - \dim[\fC^-] = 1 + \dim\ucM_z^+([\fC^-],[\fC^+]) \geq 1.
    \end{equation*}
    It follows that the complex $\CMto$ is a $\bJ$-filtered complex where the filtration is that coming from the grading $\Gr$.  The spectral sequence $\widecheck E$ is the spectral sequence associated to this filtration.
\end{proof} \begin{remark}
    In order to make sense of the filtration in the proof of Proposition \ref{prop:grading-filtration}, we are implicitly using the fact that, since $c_1(\fs_Y)$ is torsion, the action of $\Z$ on the grading set $\bJ(\fs_Y)$ is free. 
\end{remark}

Proposition \ref{prop:grading-filtration} gives a locally-finitely-generated approximation of the complex $\CMto$. We will also need an approximation of the chain map induced on $\CMto$ by a cobordism. Since the spectral sequence in Proposition \ref{prop:grading-filtration} is that associated to the filtration by the grading $\Gr$, we need to understand how gradings play with moduli spaces on a cobordism. We focus on the case of a self-cobordism, so as to avoid discussing canonical gradings. Suppose $(W,\fs_W)$ is a $\spin$ cobordism from $(Y,\fs_Y)$ to itself, where $c_1(\fs_Y)$ is torsion. Let \begin{equation*}
    d(W,\fs_W) = \frac{c_1(\fs_W)^2 - 2\chi(W) - 3\sigma(W)}4.
\end{equation*}
Then, when the moduli space $\cM([\fC^-],W_\infty,\fs_W,[\fC^+])$ is non-empty, the grading difference between critical points over $Y$ is related to the dimension of the moduli space over $W$ by \begin{equation*}
    \Gr[\fC^-] - \Gr[\fC^+] + d(W,\fs_W) = \dim\cM^+([\fC^-],W_{\infty}, \fs_W[\fC^+]) - \dim[\fC^-]
\end{equation*}
in the boundary-unobstructed case and \begin{equation*}
    \Gr[\fC^-] - \Gr[\fC^+] + d(W,\fs_W) = \dim\cM^+([\fC^-],W_\infty,\fs_W,[\fC^+]) - \dim[\fC^-] - 1
\end{equation*}
in the boundary-obstructed case. In the case that $W$ is the cobordism obtained by cutting a homology $S^1\times S^3$ along a rational homology 3-sphere $Y$ with $[Y]=1\in H^3(X;\Z)\cong \Z$, we have $d(W,\fs_W)=0$.

For a nice perturbation and a self-cobordism of $Y$ with no punctures, we can write the chain map $\Wcheck_{*}$ in matrix form as \begin{equation*}
    \Wcheck_{*} = \begin{bmatrix}
            W_o^o & -\partial_o^u\,w \Wbar_u^s \\
            W_s^o & \Wbar_s^s
        \end{bmatrix}.
\end{equation*} \begin{lemma}\label{lem:fiber-products-zero}
    Suppose $d(W,\fs_W)=0$. Then, the operator $-\partial_o^u\,w \Wbar_u^s$ is identically zero.
\end{lemma} \begin{proof}
    The operator $-\partial_o^u\,w \Wbar_u^s$ counts contributions from fiber products of moduli spaces of the form \begin{equation*}
        \dim \cM^+([\fC^-],W_\infty,\fs_W,[\fA]) \times_{[\fA]} \ucM^+([\fA],[\fC^+])
    \end{equation*}
    where $[\fC^-],[\fA],[\fC^+]$ are boundary-stable, boundary-unstable, and irreducible, respectively. Suppose, aiming for a contradiction, that such a fiber product has a non-zero contribution. Since positive-dimensional $\delta$-chains in $[\fC^+]$ are degenerate, we must have \begin{equation*}
        \begin{split}
            0 &= \dim \cM^+([\fC^-],W_\infty,\fs_W,[\fA]) \times_{[\fA]} \ucM^+([\fA],[\fC^+]) \\
            &= \dim \cM^+([\fC^-],W_\infty,\fs_W,[\fA]) + \dim \ucM^+([\fA],[\fC^+]) - 2. \\
        \end{split}
    \end{equation*}
    We must also have $\dim \ucM^+([\fA],[\fC^+])=0$, so $\dim \cM^+([\fC^-],W_\infty,\fs_W,[\fA])=2$. The moduli space $\cM^+([\fC^-],W_\infty,\fs_W,[\fA])$ is boundary-obstructed. In the boundary-obstructed case (with $d(W,\fs_W)=0$), the grading difference is related to the dimension of the moduli space by \begin{equation*}
        \Gr[\fC^-] - \Gr[\fA] = \dim \cM^+([\fC^-],W_\infty,\fs_W,[\fA]) - \dim [\fC^-] - 1 = -1.
    \end{equation*}
    This is not possible; the grading difference between a boundary-stable $\CP^1$ and a boundary-unstable $\CP^1$ lying over the same reducible downstairs must be positive. Therefore, the contribution from any such fiber product is zero.
\end{proof}
\begin{lemma}\label{lem:graded-chain-map}
    Suppose that $Y,\fs_Y,\fq$ are as in Proposition \ref{prop:grading-filtration} and that $(W,\fs_W):(Y,\fs_Y)\to (Y,\fs_Y)$ is a $spin$ cobordism such that $H_*(W;\Q)\cong H_*(S^3;\Q)$. Then, for a compatible perturbation over $W$, the chain map $\Wcheck_{*}$ induces a natural family of maps $\widecheck E_{*,*}^k \to \widecheck E_{*,*}^k$ converging to the induced map on $\HMto(Y,\fs_Y)$.
\end{lemma} \begin{proof}
    It suffices to prove that $\Wcheck_{*}$ takes chains with $\Gr = j$ to chains with $\Gr \leq j$. Suppose the chain map $\Wcheck_{*}$ has a non-zero component going from $[\fC^-]$ to $[\fC^+]$. The assumption on the homology of $W$ implies $d(W,\fs_W)=0$. By Lemma \ref{lem:fiber-products-zero}, there is no contribution from the upper-right corner of the matrix for $\Wcheck_{*}$, so any non-zero contribution must come from a moduli space \begin{equation*}
        \cM^+([\fC^-],W_\infty,\fs_W,[\fC^+]) 
    \end{equation*}
    where either $[\fC^-]$ is irreducible or both $[\fC^-],[\fC^+]$ are boundary-stable.
    
    The grading difference between $[\fC^-]$ and $[\fC^+]$ is related to the dimension of the moduli space by \begin{equation*}
        \Gr[\fC^-] - \Gr[\fC^+] = \dim \cM^+([\fC^-],W_\infty,\fs_W[\fC^+]) - \dim [\fC^-].
    \end{equation*}
     If $[\fC^-]$ is irreducible, then we automatically have $\Gr[\fC^+] \leq \Gr[\fC^-]$, since $\dim[\fC^-]=0$. In the case that both $[\fC^-],[\fC^+]$ are boundary-stable, we automatically have $\Gr[\fC^+] \leq \Gr[\fC^-]+2$, but the grading difference $\Gr[\fC^-] - \Gr[\fC^+]$ must be a multiple of 4, and so $\Gr[\fC^+] \leq \Gr[\fC^-]$.
\end{proof}

We record the following lemma for future reference. \begin{lemma}\label{lem:identity-on-reducibles}
    Suppose $W,\fs_W$ are as in Lemma \ref{lem:graded-chain-map}. Then, the restriction of $\Wcheck_{*}$ to the boundary stable reducibles $C^s\subseteq \CMto(Y,\fs_Y)$ is a direct sum of maps $C_*^{\delta}([\fC];\Z_{[\fC]})\to C_*^{\delta}([\fC];\Z_{[\fC]})$, each chain-homotopic to the identity. 
\end{lemma}\begin{proof}
    The proof is the same as the proof of \cite[Lemma~3.3]{lin-ruberman-saveliev}.
\end{proof} 

}

\subsection{Compactness and Gluing}\label{subsection:necks}{

\newcommand{\XT}{\tildeX_{T}}

The double cover $\tildeX$ decomposes as two copies of $W$ glued along opposite copies of $Y$. Let $i^i,i^o:Y\hookrightarrow W$ be the inclusions of the two (incoming, outgoing) boundary components. For bookeeping purposes, let $W^0,W^1$ denote the two copies of $W$ and let $Y^0,Y^1$ be the two boundary components of $W$. Then, \begin{equation*}
    \tildeX \cong \colim\left(\begin{tikzcd}[sep=small]
                        & W^0 & \\
        Y^0\ar[ur, "i^i"]\ar[dr, "i^o", swap] & & Y^1 \ar[ul, "i^o", swap]\ar[dl, "i^i"] \\
                        & W^1 &  
    \end{tikzcd}\right).
\end{equation*}

Suppose that $\tildeX$ is equipped with an $\iota$-invariant metric, i.e., a metric with identical restrictions to the two copies of $W$. Let $W_{T} = ([-T,0]\times Y) \cup W \cup ([0,T]\times Y)$ be the Riemannian 4-manifold obtained from $W$ by sticking cylinders onto the ends and let $\XT$ be the Riemannian 4-manifold obtained from $X$ by sticking copies of the cylinder $Z_T = [-T,T]\times Y$ in place of each of the copies of $Y$, i.e., \begin{equation*}
    \XT \cong \colim\left(\begin{tikzcd}[sep=small]
            & W^0_T & \\
        Y^{0} \ar[ur, "i^i"]\ar[dr, "i^o", swap] & & Y^{1} \ar[ul, "i^o", swap]\ar[dl, "i^i"] \\
                        & W^1_T &  
    \end{tikzcd}\right).
\end{equation*}
We allow for the case $T=\infty$, taking $Z_\infty = ([-\infty,0)(0,+\infty])\times Y$ and $W_{\infty} = ((-\infty, 0]\times Y) \cup W \cup ([0,+\infty)\times Y)$. The proof of \cite[Theorem~8.2]{lin-ruberman-saveliev} goes through mostly unchanged to prove the following compactness result. \begin{proposition}\label{prop:compactness}
    For every sequence of positive real numbers $T_n\to+\infty$ and every sequence of monopoles $[A_n,\Phi_n]\in\cM(\tildeX_{T_n})$, there exists a subsequence $[A_{n_k},\Phi_{n_k}]$, gauge transformations $u_{n_k}$, and a monopole $[A_\infty,\Phi_\infty]\in\cM(\tildeX_\infty) = \cM(W_\infty)\times \cM(W_\infty)$ such that $u_{n_k}\cdot (A_{n_k},\Phi_{n_k})$ converges to $(A_\infty,\Phi_\infty)$.
\end{proposition}

Suppose that for each $T$, the relevant gauge transformation $u$ is chosen so that its restriction to $W^0\sqcup W^1$ is independent of $T$ and such that its restrictions to the two cylinders $Z_T^0,Z_T^1$ are constant with values $1,-1$, respectively.  Suppose $0<T'<T\leq+\infty$. The irreducible moduli space on $\XT$ decomposes as \begin{equation}\label{eq:big-lim}
    \cM^*(\XT) \cong \lim\left(\begin{tikzcd}[sep=small]
            \cB^*(Y^{00}) & \ar[l,"R^i", swap] \cM^*(W_{T'}^0) \ar[r, "R^o"] & \cB^*(Y^{01}) \\
            \cM^*(Z_{T-T'}^0) \ar[u, "R^o"]\ar[d, "R^i", swap] &   & \cM^*(Z_{T-T'}^1) \ar[u, "R^i", swap]\ar[d, "R^o"] \\
            \cB^*(Y^{10}) & \cM^*(W_{T'}^1) \ar[l,"R^o"]  \ar[r, "R^i", swap] & \cB^*(Y^{01})
    \end{tikzcd}\right)
\end{equation}
where $R^i,R^o$ are the restrictions induced by boundary inclusions. The involution $I$ on $\cM^*(\XT)$ acts on the diagram in \eqref{eq:big-lim} by rotating the rectangle through $180^\circ$ and applying involutions $I$ to each of the objects. On each copy of $\cB^*(Y)$, this involution is $\jmath$. The fixed-point set $\cM^*(\XT)^I$ is identified with a fiber product of a copy of $\cM^*(W_{T'})$ and a copy of $\cM^*(Z_{T-T'})$, \begin{equation*}
    \begin{split}
        \cM^*(\XT)^I &\cong Fib((R^i,\,R^o)_{W\to Y}, (R^o,\,\jmath R^i)_{I\to Y}) \\ &=  \lim\left(\begin{tikzcd}[ampersand replacement=\&]
        \cM^*(W_{T'}) \ar[r, "{(R^i,\,R^o)}"] \& \cB^*(Y)\times \cB^*(Y) \& \ar[l, "{(R^o,\,\jmath R^i)}", swap] \cM^*(Z_{T-T'}) \end{tikzcd}\right).   
    \end{split} 
\end{equation*} 
Lemma \ref{lem:irreducible-moduli-spaces} tells us that the irreducible moduli spaces $\cM_\Re^*(\XT)$ and $\cM^*(\XT)^I$ coincide. The main result we will need in order compare the count of real solutions and the count of solutions on $W_\infty$ is the following proposition. \begin{proposition}\label{prop:moduli-spaces-homeo}
    Assume that the spin Dirac operator on $W_\infty$ is invertible. There exists $T_0>0$ such that for all $T\geq T_0$, the real irreducible moduli space $\cM_\Re^*(\XT)$ is regular and there exists a homeomorphism \begin{equation*}
        \cM_\Re^*(\XT) = \cM^*(\XT)^I \xlongrightarrow{\sim} \coprod_{[\fa]\in\Cr^{ o}}\cM^*([\fa],W_\infty,\fs_W,\jmath[\fa]).
    \end{equation*}
\end{proposition} The proof of Proposition \ref{prop:moduli-spaces-homeo} is similar to that of \cite[Theorem~9.1]{lin-ruberman-saveliev}. Let $u_{[\fa]}(T,-): B([\fa])\to \cM^*(Z_T)$ be the chart provided by \cite[Theorem~9.3]{lin-ruberman-saveliev} and let $\mu_{[\fa],T}^\jmath = (R^o,\jmath R^i)\circ u_{[\fa]}(T,-):B([\fa])\to \cB^*(Y)\times \cB^*(Y)$. The following lemma serves the purpose of \cite[Lemma~9.6]{lin-ruberman-saveliev}. \begin{lemma}\label{lem:fiber-product-regular}
    There exist $T''>T'>0$ such that for all $T\geq T''$ sufficiently large, there exists a homeomorphism \begin{equation*}
        \cM_\Re^*(\XT) \xrightarrow{\sim} \coprod_{[\fa]\in\Cr^ o} Fib((R^i,R^o)_{W\to Y}, \mu_{[\fa],T}^\jmath).
    \end{equation*}
    Moreover, the real irreducible moduli space is regular if and only if the maps $(R^i,R^o)_{W\to Y}$ and $\mu_{[\fa],T}^\jmath$ are transverse for all $[\fa]\in\Cr^ o$.
\end{lemma} \begin{proof}
    The existence of a homeomorphism is proved in the same way as in \cite[Lemma~9.6]{lin-ruberman-saveliev}. The real part of the linearized equations over $\XT$ restricts to the linearized ordinary equations over $W_{T'}$ and $Z_{T-T'}$. An argument based on \cite[Proposition~17.2.8]{kronheimer-mrowka} proves the regularity claim. See also \cite[Theorem~19.1.4]{kronheimer-mrowka}.
\end{proof}
\begin{proof}[Proof of Proposition \ref{prop:moduli-spaces-homeo}]
    The proof is the same as the proof of \cite[Theorem~9.1]{lin-ruberman-saveliev}, with Proposition \ref{prop:compactness} filling the role of \cite[Theorem~8.2]{lin-ruberman-saveliev} and Lemma \ref{lem:fiber-product-regular} filling the role of \cite[Lemma~9.6]{lin-ruberman-saveliev}.
\end{proof}

}

\subsection{Comparing Signs}\label{subsection:comparing-signs}{

\newcommand{\XT}{\tildeX_{T}}

Proposition \ref{prop:moduli-spaces-homeo} establishes a correspondence between real irreducibles on $\tildeX$ and irreducibles on the cobordism $W$ joining an irreducible critical point $[\fa]$ to its image under $\jmath$. In order to compare the invariants obtained by counting solutions, we need to compare orientations of the moduli spaces. By understanding how each of the moduli spaces is oriented and comparing these orientations, we will be able to express the invariant $|\deg|$ in terms of a Lefschetz number. Before passing to homology, we need to work at the chain level. The complex $\CMto$ decomposes as a direct sum $\CMto = C^o\oplus C^s$. Note that $C^o \cong \bigoplus_{[\fC]\in\Cr^o} \Z_{[\fC]}$ is a finitely-generated $\bJ$-graded abelian group, so the Lefschetz number of a map $C^o\to C^o$ is well-defined using the canonical map $\bJ\to \Z/2$. Let $\operatorname{proj} : \CMto \to C^o$ be the projection with respect to the aforementioned direct-sum decomposition.

\begin{proposition}\label{prop:comparing-signs}
    Let $W_{*}^o:C^ o\to C^ o$ be the composition \begin{equation*}
        \begin{tikzcd}
            C^ o \ar[r, hook] & \CMto \ar[r, "\Wcheck_{*}"] & \CMto \ar[r,"{\operatorname{proj}}"] & C^o.
        \end{tikzcd}
    \end{equation*}
    Then, \begin{equation*}
        |\deg(\tildeX)| = |1 + 2\Tr(\jmath_*W_*^o)|.
    \end{equation*}
\end{proposition}

It will be easier to directly compare orientations if we recall Remark \ref{rmk:pin2-structures}. For all $T\in [0,+\infty)$, the real irreducible moduli space $\cM_\Re^*(\XT)$ may be identified with the $\spin^{c-}$ moduli space $\cM^*(X_T, I\setminus\fs_{\tildeX})$. There is an identification of the linearized Seiberg-Witten operators relevant to orienting these spaces, so that $\#\cM_\Re^*(\XT) = \#\cM^*(X_T, I\setminus\fs_{\tildeX})$. For $T=+\infty$, we have an identification of the disjoint union $\coprod_{[\fa]\in\Cr^{ o}}\cM^*([\fa],W_\infty,\jmath[\fa])$ with the disjoint union $\coprod_{[\fa]\in\Cr^{ o}}\cM^*(X_\infty,I\setminus\fs_{\tildeX};[\fa])$ of moduli spaces of $\spin^{c-}$ solutions on the cylindrical ends manifold $X_\infty$ limiting to the same $[\fa]$ at both ends. The limiting points of $\spin^{c-}$ solutions are not quite well-defined, as there is an ambiguity as to whether the limiting point is $[\fa]$ or $\jmath[\fa]$; nevertheless, the union over all possible limits gives a well-defined moduli space. 

\begin{proof}[Proof of Proposition \ref{prop:comparing-signs}]
Let $\cM_T^* = \cM^*(X_T, I\setminus\fs_{\tildeX})$ for $T\in [0,+\infty)$ and let $\cM_\infty^*=\coprod_{[\fa]\in\Cr^{ o}}\cM^*(X_\infty,I\setminus\fs_{\tildeX};[\fa])$. The homeomorphisms of Proposition \ref{prop:moduli-spaces-homeo} can be patched together to give a fiber bundle $\cM^*\to [T_0,+\infty]$ such that the fiber over $T$ is $\cM_T^*$. We get a line bundle over $\cM^*$ by patching together the determinant line bundles $\det(\cD_T)$ where $\cD_T$ is the linearized Seiberg-Witten operator over $\cM_T^*$. This establishes a correspondence between orientations of $\cM_T^*$ for $0 \ll T<+\infty$ and orientations of $\cM_\infty^*$. For finite $T$, the moduli space $\cM_T^*$ is oriented such that \begin{equation*}
    |\deg(\tildeX)| = |1\pm 2\#\cM_T^*|.
\end{equation*}
We call the corresponding orientation on $\cM_\infty^*$ the $\fo_{closed}$. There is also the canonical orientation $\fo_{can}$ for which $\#_{\fo_{can}}\cM_\infty^* = \operatorname{trace}(\jmath_*W_*^o).$
For fixed $[\fa]\in\Cr^o$, these two orientations differ by the sign $(-1)^{\gr^{(2)}[\fa]}$. This sign comes from the definition of the duality map identifying orientations over $Y$ and $-Y$; see \cite[Definition~22.5.3, Proposition~25.2.1]{kronheimer-mrowka}. This establishes the equality \begin{equation*}
    |\deg(\tildeX)| = |1\pm 2\Tr(\jmath_*W_*^o)|.
\end{equation*}
By Lemma \ref{lem:identity-on-reducibles}, we deduce that the orientation on $\cM_\Re^{f}(\XT)$ such that the real reducible is counted with positive sign corresponds to the orientation $\fo_{closed}$ such that $\#_{\fo_{closed}}\cM_\infty^* = \Tr(\jmath_*W_*^o).$ Therefore, \begin{equation*}
    |1 + 2\Tr(\jmath_*W_*^o)|.
\end{equation*}
\end{proof}

}

\subsection{Proof of Theorem \ref{thm:deg-formula}}\label{subsection:the-proof}{

Proposition \ref{prop:comparing-signs} expresses $|\deg(\tildeX)|$ in terms of the Lefschetz number on irreducibles. It remains to interpret this as a Lefschetz number on $\HMred(Y)$. We view $\HMred(Y)$ as the cokernel of the map $i_*:\HMbar(Y)\to \HMto(Y)$. Let $N\gg 0$ and let ${\CMto}_{\leq N}$ be the subcomplex of $\CMto(Y)$  generated by chains in critical submanifolds with grading $\leq N$, and let $\HMto_{\leq N}$ be the homology of this subcomplex. For $N$ sufficiently large, we have a short exact sequence \begin{equation*}
    0 \longrightarrow \im i_{*\leq N} \longrightarrow \HMto_{\leq N} \longrightarrow \HMred \longrightarrow 0.
\end{equation*}
Therefore, \begin{equation}\label{eq:trace1}
    \Tr\Big(\jmath_* \Wcheck_{\leq N}\Big) = \Tr\Big(\jmath_* \Wcheck_{\leq N}|_{\im i_*}\Big) + \Tr\Big(\jmath_* W_{*}\Big).
\end{equation}
Let $\widecheck{E}$ be the spectral sequence given by Proposition \ref{prop:grading-filtration}. Lemma \ref{lem:graded-chain-map} implies that $W$ induces a chain map on \begin{equation*}
    \widecheck E_{ij}^1 = \bigoplus_{\underset{\operatorname{Gr}([\fC]) = j}{[\fC]\in \Cr^o\sqcup\Cr^s}} H_i([\fC];\Z_{[\fC]})
\end{equation*}
such that the Lefschetz number of the induced map on the truncation $\widecheck{E}^{1}_{\leq N}$ is the same as on the homology $\HMto_{\leq N}$. Call this map $\widecheck{EW}_{\leq N}$. We have \begin{equation*}
    \Tr\Big(\jmath_* \Wcheck_{\leq N}\Big) = \Tr\Big(\jmath_* \widecheck{EW}_{\leq N}\Big) = \Tr\Big(\jmath_* W_*^o\Big) + \Tr\Big(\jmath_* EW_{\leq N}^{s}\Big).
\end{equation*}
Since the reducible critical submanifolds are all homeomorphic to $S^2$ with $\jmath_*$ acting as the antipodal map, the Lefschetz number of $\jmath_*$ on reducibles is zero. By Lemma \ref{lem:identity-on-reducibles}, we have $\Tr(\jmath_* W_{\leq N}^{s}) = 0$. Therefore, \begin{equation}\label{eq:trace2}
   \Tr\Big(\jmath_* \Wcheck_{\leq N}\Big) = \Tr\Big(\,\jmath_* W_*^o\,\Big).
\end{equation}
Combining \eqref{eq:trace1} and \eqref{eq:trace2} and the fact that $W$ induces the identity on $\im i_*^{\leq N}$, we have \begin{equation}\label{eq:trace3}
    \Tr\Big(\,\jmath_* W_*^o\,\Big) = \Tr\Big(\jmath_* \text{ on }\im (i_*)_{\leq N}\Big) + \Tr\Big(\jmath_* W_*\text{ on }\HMred\Big)
\end{equation}
Consider the case that $W$ is the trivial cobordism, so that $W$ induces the identity on irreducibles as well. The Lefschetz number of $\jmath_*$ on irreducibles is zero since $\jmath_*$ freely permutes irreducibles. Therefore, as a special case of \eqref{eq:trace3}, we have \begin{equation}\label{eq:trace4}
    0 = \Tr\Big(\jmath_* \text{ on }\im (i_*)_{\leq N}\Big) + \Tr\Big(\jmath_* \text{ on }\HMred\Big).
\end{equation}
Combining \eqref{eq:trace3} and \eqref{eq:trace4}, \begin{equation*}
    \Tr\Big(\,\jmath_* W_*^o\,\Big) = \Tr\Big(\jmath_*(W_* - 1) \text{ on }\HMred\Big).
\end{equation*}
This completes the proof.

}


\section{Examples} \label{section:examples}{

Our main theorem applies to smooth 2-knots in $S^4$ which bound rational homology balls. The set of smooth 2-knots in $S^4$ is far from classified and likely far from classifiable. Even if we restrict attention to the class of 2-knots which bound rational homology balls, there is likely no hope of classification. Further, it is not clear whether this is a particularly large class.

The most famous and well-studied class of 2-knots consists of the spun knots and their cousins-in-law, the deform-spun knots; see \cite{litherland}. Suppose $K\hookrightarrow S^3$ is a knot and $\gamma$ an untwisted deformation of $K$. If $m\neq 0$, then the twisted deform-spun knot $\tau^m\gamma K \hookrightarrow S^4$ is fibered. The fiber is typically, but not always, a rational homology ball. When $\gamma$ is the identity, the fiber of the twist-spin $\tau^m K$ is the $m$-fold cyclic branched cover $\Sigma_m(K)$ and when $\gamma=\rho^n$ is a roll, the fiber of the twist-roll-spin $\tau^m \rho^n K$ is a Dehn surgery along a nullhomologous knot in $\Sigma_m(K)$ with nonzero framing. As a consequence of \cite[\textsection 6]{fox-iii}, the fiber of $\tau^m\rho^n K$ is a rational homology ball if and only if the Alexander polynomial $\Delta_K$ does not have zeroes at any $m$-th roots of unity.

For pure untwisted deformations (i.e., when $m=0$), the corresponding deform-spun knots typically do not bound rational homology balls. Primary topological obstructions to a 2-knot bounding a rational homology ball come from the Alexander module. Recall \cite[\textsection 3.3.1]{carter-saito} that the \emph{(first) Alexander module} $A(\cK)$ of a knot $\cK$ in any dimension is the relative first homology of the infinite cyclic cover, $H_1(X_\infty(K),\pi^{-1}(x); \Z)$, considered as a module over the ring $\Lambda_\Z=\Z[T^{-1},T]$ of Laurent polynomials. If $\cK$ bounds a rational homology ball, then $A(\cK)\otimes \Q = 0$. Let $\Delta(\cK)\subseteq \Lambda_\Z$ denote the \emph{Alexander ideal} of $\cK$, i.e., the first elementary ideal of $A(\cK)$. Then, $A(\cK)\otimes \Q = 0$ if and only if $\Delta(\cK)\otimes \Q = \Lambda_\Q$. For a knot $K\hookrightarrow S^3$, the ideal $\Delta(K)$ is principal, generated by the Alexander polynomial $\Delta_K(T)$. The Alexander polynomial of $K$ is monic. Thus, for 1-knots, the condition that $A(K)$ is $\Z$-torsion amounts to $\Delta(K)=1$. 

Let us see how the Alexander module obstructs twist-roll-spun knots from bounding rational homology balls. We start with the purely roll-spun knots.

\begin{lemma}\label{lem:alexander-roll}
    Let $K\hookrightarrow S^3$ be a knot. Then, for all $n\in \Z$, \begin{equation*}
        A(\rho^n K) = A(K).
    \end{equation*}
\end{lemma} \begin{proof}
    Let $\mu_K,\lambda_K\subseteq \partial (S^3\setminus K)$ denote a meridian, longitude of $K$, respectively. Up to homotopy, $S^4\setminus \rho^n K$ can be built from $S^1\times (S^3\setminus K) = S^1\times \partial(S^3\setminus K) \cong S^1\times \mu_K\times \lambda_K$ by gluing $S^1\times D^2$ along $S^1\times \partial D^2$ to $\partial (S^1\times (S^3\setminus K))$ such that $S^1\times *$ is glued to $* \times \mu_K$ and $*\times\partial D^2$ is glued to $S^1\times * + n(*\times \lambda_K)$. The infinite cyclic cover $X_\infty(\rho^n K)$ can be built up to equivariant homotopy from $S^1\times X_\infty(K)$ by gluing $\R\times D^2$ along $\R\times \partial D^2$ to $\partial(S^1\times X_\infty(K)) = S^1\times \partial X_\infty(K) \cong S^1 \times \R \times \lambda_K$ such that $\R\times *$ is glued to $*\times \R$ and $*\times \partial D^2$ is glued to $S^1\times * + n(*\times \lambda_K)$.

    At the level of homology, this gluing has the effect of killing $[S^1\times * + n(*\times \lambda_K)] \in H_1(S^1\times X_\infty(K);\Z)$. Since $\lambda_K$ is already nullhomologous in $X_\infty(K)$, the homology $H_1(X_\infty(\rho^n K);\Z)$ exhibits no dependence on $n$ as a $\Lambda_\Z$-module. By the five lemma and the long exact sequence for the pair $X_\infty(\rho^nK),\pi^{-1}(x)$, there is likewise no dependence on $n$ for the Alexander module $A(\rho^n K)$. In the case $n=0$, we have the spun knot $\rho^0 K = \spin(K)$. Recall that the Alexander module can be computed from the knot group using Fox's free differential calculus. The knot groups for $K$ and $\spin(K)$ are identical. Therefore, so are their Alexander modules.
\end{proof} \begin{corollary}
    Let $K\hookrightarrow S^3$ be a knot. If $\Delta(K)\neq 1$, then for all $n\in\Z$, the roll-spin $\rho^n K$ does not bound a rational homology ball.
\end{corollary} 
Now, we move on to twist-roll-spins.
\begin{proposition}\label{prop:alexander-twist-roll}
    Let $K\hookrightarrow S^3$ be a knot. For all $m,n\in \Z$ such that $m\neq 0$, the twist-roll-spin $\tau^m\rho^n K$ bounds a rational homology ball if and only if the fiber is itself a rational homology ball, i.e., if and only if the Alexander polynomial $\Delta_K$ does not have any zeroes at $m$-th roots of unity.
\end{proposition}\begin{proof}
    Wihtout loss of generality, assume $m>0,n\geq 0$. We can use Fox's free differential calculus to relate the Alexander ideals $\Delta(\tau^m\rho^n K)$ and $\Delta(\rho^n K)$, the latter of which is the same as $\Delta(K)$ by Lemma \ref{lem:alexander-roll}. Consider a Wirtinger presentation for the knot group $\pi K$, say \begin{equation*}
        \pi K = \langle x_1,\ldots, x_M | r_1,\ldots, r_N \rangle.
    \end{equation*}
    A Wirtinger presentation for the knot group $\pi(\tau^m\rho^n K)$ is then \begin{equation*}
        \pi(\tau^m\rho^n K) = \langle x_1,\ldots, x_M | r_1,\ldots, r_N, [x_1^m \lambda^n, x_1], \ldots, [x_1^m \lambda^n, x_M] \rangle
    \end{equation*}
    where $\lambda\in F_M$ is a word representing a longitude of $K$. Let $\partial_i$ denote the Fox derivative on the free group $F_M$ with respect to the generator $x_i$. Write $[g]$ for the image of an element $g\in \Z F_M$ under the exponent-sum map $\Z F_M \to \Lambda_\Z$ (which factors through the abelization map $\Z\pi K\to \Lambda_\Z$).

    Let $w = x_1^m\lambda^n$. Differentiating the commutator $[w, x_j]$ yields \begin{equation}\label{eq:commutator-derivative}
        \begin{split}
            \partial_i [w, x_j] = (1 - wx_jw^{-1})\partial_iw + \delta_{ij}(w - [w,x_j]).
        \end{split}
    \end{equation}
    The image of $\partial_i w$ in $\Lambda_\Z$ is $\delta_{i1}(1+T+T^2+\ldots+T^{m-1}) + nT^m[\partial_i \lambda]$. Therefore, \eqref{eq:commutator-derivative} maps to \begin{equation*}
        \begin{split}
            [\partial_i [w, x_j]] &= (1-T)(\delta_{i1}(1+T+T^2+\ldots+T^{m-1}) + nT^m[\partial_i \lambda]) + \delta_{ij} (T^m - 1) \\
            &= (\delta_{ij}-\delta_{i1})(T^m-1) + nT^m(T-1)[\partial_i \lambda].
        \end{split}
    \end{equation*}
    Let $f(T) = T^m(1-T)$. An Alexander matrix for $\tau^m\rho^n K$ is \begin{equation}\label{eq:alexander-matrix-1} { \small
        \begin{bNiceMatrix}
            \Block{4-5}{R} & & & & \\
            & & & & \\
            & & & & \\
            & & & & \\

            nf(T)[\partial_i\lambda] & nf(T)[\partial_2\lambda] & nf(T)[\partial_3\lambda] & \cdots & nf(T)[\partial_M\lambda] \\
            -(T^m-1) + nf(T)[\partial_1\lambda] & (T^m-1) + nf(T)[\partial_2\lambda] & nf(T)[\partial_3\lambda] & \cdots & nf(T)[\partial_M\lambda] \\
            -(T^m-1) + nf(T)[\partial_1\lambda] & nf(T)[\partial_2\lambda] & (T^m-1) + nf(T)[\partial_3\lambda] & \cdots & nf(T)[\partial_M\lambda] \\
            \vdots & \vdots & \vdots & \ddots & \vdots \\
            -(T^m-1) + nf(T)[\partial_1\lambda] & nf(T)[\partial_2\lambda] & nf(T)[\partial_3\lambda] & \cdots & (T^m-1) + nf(T)[\partial_M\lambda] \\
        \end{bNiceMatrix} }
    \end{equation}
    where the block $R$ is an Alexander matrix for $K$. By definition, the Alexander ideal $\Delta(\tau^m\rho^n K)$ is generated by the $(M-1)\times (M-1)$ minors of \eqref{eq:alexander-matrix-1}. By elementary row operations, we can transform \eqref{eq:alexander-matrix-1} into \begin{equation}\label{eq:alexander-matrix-2}
        \begin{bNiceMatrix}
            \Block{4-5}{R} & & & & \\
            & & & & \\
            & & & & \\
            & & & & \\
            
            nf(T)[\partial_i\lambda] & nf(T)[\partial_2\lambda] & nf(T)[\partial_3\lambda] & \cdots & nf(T)[\partial_M\lambda] \\
            -(T^m - 1) & (T^m - 1) & 0 & \cdots & 0 \\
            -(T^m-1)  & 0 & (T^m-1) & \cdots & 0 \\
            \vdots & \vdots & \vdots &\ \ddots & \vdots \\
            -(T^m-1) & 0 & 0 & \cdots & (T^m-1) 
        \end{bNiceMatrix}.
    \end{equation}
    Observe that the block \begin{equation}\label{eq:alexander-block}
        \begin{bNiceMatrix}
            \Block{4-5}{R} & & & & \\
            & & & & \\
            & & & & \\
            & & & & \\

            nf(T)[\partial_i\lambda] & nf(T)[\partial_2\lambda] & nf(T)[\partial_3\lambda] & \cdots & nf(T)[\partial_M\lambda] \\
        \end{bNiceMatrix}
    \end{equation}
    is an Alexander matrix for $\rho^n K$. By Lemma \ref{lem:alexander-roll}, the matrix \eqref{eq:alexander-block} is also an Alexander matrix for $K$. Since all minors of \eqref{eq:alexander-matrix-2} are either minors of \eqref{eq:alexander-block} or multiples of $T^m-1$, we have an inclusion of ideals $\Delta(\tau^m\rho^n K)\subseteq (\Delta_K(T),T^m-1)$. Suppose that $\Delta_K(T)$ has a zero at a complex root of unity. Then, $(\Delta_K(T), T^m-1)\otimes \Q$ is a proper ideal of $\Lambda_\Q$, and therefore so is $\Delta(\tau^m\rho^n K) \otimes \Q$. It follows that $\tau^m\rho^n K$ does not bound a rational homology ball. Conversely, if $\Delta_K(T)$ has no zeroes at complex roots of unity, then the fiber is a rational homology ball as discussed earlier in this section.
\end{proof} 

\begin{remark}
    Given a particular 2-knot $\cS\hookrightarrow S^4$ bounding a Seifert solid $Y^\circ\hookrightarrow S^4$, we can build other 2-knots bounding embeddings of $Y^\circ$ by performing surgeries on $S^4 \setminus Y^\circ$ which cancel as surgeries on $S^4$, in a way that is reminiscient of constructions of knot concordances.

    If we only perform 0- and 1-surgeries, the resulting 2-knot is ribbon-concordant to $\cS$. Such a surgery can be specified up to framing by an integer $n$ and a choice of elements \begin{equation*}
        g_1,\ldots,g_n\in \pi_1((S^4 \setminus Y^\circ ) \# (S^1\times S^3)^{\# n}) = \pi_1(S^4 \setminus Y) * F_n
    \end{equation*}
    mapping to a basis of the free group $F_n$ under the inclusion-induced map \begin{equation*}
        \pi_1((S^4 \setminus Y^\circ) \# (S^1\times S^3)^{\# n}) \to \pi_1(S^4\# (S^1\times S^3)^{\# n}) = F_n.
    \end{equation*}
\end{remark}

Having discussed some examples of 2-knots that bound rational homology balls and some that do not, let us turn to computations of $|\deg|$ using Theorem \ref{thm:deg-formula}. The 4-dimensional invariant $|\deg|$ is closely related to an invariant of 3-dimensional knots which goes by the same name. According to \cite[Proposition~4.30]{miyazawa}, if $K\hookrightarrow S^3$ is a knot, then $|\deg(\tau^m\rho^n K)|=|\deg(K)|$ for all $m,n\in\Z$ such that $m+2n\equiv 2\pmod 4$. As a direct application of Theorem \ref{thm:deg-formula}, we can compute the 3-dimensional invariant $|\deg(K)|$ where $K$ is a Montesinos knot or a torus knot, by way of the 2-twist-spin $\tau^2 K$. Note that these computations can already be done in other ways, for example using the formulas of \cite{kang-park-taniguchi}.

If $\cS = \tau^2 K$, then $\tildeX(\cS) = S^1 \times\Sigma_2(K)$. The deck transformation on $\tildeX(\cS)$ rotates the $S^1$ factors halfway and applies the 3-diemensional deck transformation to the $\Sigma_2(K)$ factor.

\begin{example}[Montesinos Knots]\label{ex:montesinos}
    Suppose that $Y=M(b_0,(a_1,b_1),\ldots,(a_n,b_n))$ is a Seifert-fibered rational homology sphere, which we may view as an orbifold circle bundle over the cone orbifold $\cO = S^2(a_1,\ldots,a_n)$. Suppose $Y$ is equipped with its Thurston geometry. Suppose further that $|H_1(Y;\Z)|$ is odd, and equip $Y$ with its unique $\spin$ structure $\fs_Y$.
    
    Position the cone points of $\cO$ along the equator of $S^2$. Note that this implicitly fixes an order of the cone points up to cyclic and anti-cyclic permutations. Then, complex conjugation on $\CP^1$ induces an orientation-reversing involution of $\cO$ which fixes the cone points. This involution lifts to an orientation-preserving involution $\iota_Y$ (the \emph{Montesinos involution}) on $Y$ which reverses the circle action, realizing $Y$ as the double branched cover of a Montesinos link in $S^3$. The involution $\iota_Y$ in turn induces an involution $\iota_{Y*}$ on the various flavors of monopole Floer homology (preserving the part coming from the $\spin$ structure), which turns out to be the same as the complex conjugation involution $\jmath_*$. For Seifert fibered manifolds, the involution $\iota_{Y*}=\jmath_*$ can be understood geometrically at the level of spinors using the description of the moduli spaces in \cite{mrowka-ozsvath-yu}, or combinatorially as the involution induced by reflection in a graded root as in \cite{dai}
    
    Consider the mapping torus $X = S^1\times_{\iota_Y} Y$. The unique double cover of $X$ is $\tildeX = S^1 \times Y$ with deck transformation $\iota_{\tildeX}(e^{i\theta},y) = (e^{i(\theta+\pi)},\iota_Y(y))$.  Since $X$ is a homology $S^1\times S^3$ with a separating rational homology sphere, we can apply Theorem \ref{thm:deg-formula} to compute $|\deg(\tildeX)|$. The cobordism $W$ is the cylinder $I\times Y$ twisted by the involution $\iota_Y$. We have $W_*=\iota_{Y*}=\jmath_*$ and so $\jmath_*(W_*-1)=1-\iota_{Y*}$. If we pass to $\R$ coefficients, the Lefschetz number of $1-\iota_{Y*}$ is twice the Euler characteristic of the $\iota_{Y*}$-anti-invariant part. Theorem \ref{thm:deg-formula} gives \begin{equation*}
        |\deg(\tildeX)| = |1 + 4\chi(\HMred(Y,\fs_Y;\R)^{-\iota_{Y*}})|
    \end{equation*}

    As a specific family of examples, consider the Brieskorn homology spheres $Y=\Sigma(2,3,r)$ where $(6,r)=1$. The (in this case unique) Montesinos link associated to $Y$ is a knot $K=K(2,3,r)$, and $|\deg(\tildeX)|=|\deg(K)|$, where for a knot $K$, $|\deg(K)|$ is defined to be the count of real Seiberg-Witten solutions over the double branched cover with its unique real $\spin$ structure. If $r=12k+1$ or $12k+5$ where $k\geq 0$, then $\HMred(Y,\fs_Y)$ is free of rank $2k$ concentrated in even $\Z/2$ grading and $\iota_{Y*}$ acts by freely permuting a basis. We have $\chi(\HMred(Y,\fs_Y;\R)^{-\iota_{Y*}})=k$ and therefore \begin{equation*}
        |\deg(K(2,3,12k+1\text{ or }12k+5))| = 4k+1.
    \end{equation*}
    If $r=12k-1$ or $12k-5$, then $\HMred(Y,\fs_Y)$ is free of rank $2k-1$, concentrated in odd $\Z/2$ grading. There is an isomorphism of $\HMred(Y,\fs_Y)$ with $\Z^{2k}/\langle(1,1,\ldots,1)\rangle$ such that $\iota_{Y*}$ is induced by an involution of $\Z^{2k}$ which freely permutes the standard basis. We have $\chi(\HMred(Y,\fs_Y;\R)^{-\iota_{Y*}})=-k$ and therefore \begin{equation*}
        |\deg(K(2,3,12k-1\text{ or }12k-5))| = 4k-1.
    \end{equation*}
    This family of examples was worked out in a different way in \cite{kang-park-taniguchi}.

\end{example}\begin{example}[Torus Knots]
    Suppose that $p,q$ are odd coprime integers and consider the torus knot $T(p,q)$. The double branched cover of $T(p,q)$ is the Brieskorn homology sphere $\Sigma(2,p,q)$. Whereas in the previous example the involution on $Y$ was given by complex conjugation, in this example the involution $\iota_Y$ comes from the circle action. The induced automorphism $\iota_{Y*}$ of $\HMred(Y,\fs_Y)$ is in this case the identity.

    Again, we may consider the mapping torus $X=S^1\times_{\iota_Y} Y$ and its double cover $\tildeX$. In this case, $W_*=\iota_{Y*} = 1$, so Theorem \ref{thm:deg-formula} gives \begin{equation*}
        |\deg(T(p,q))| = |\deg(\tildeX)| = |1 + \Tr(\jmath_*(1-1))| = 1.
    \end{equation*}
    Again, this family of examples was worked out in a different way in \cite{kang-park-taniguchi}.
\end{example}

By Example \ref{ex:montesinos}, $|\deg(K(2,3,r))|=1$ if and only if either $r=1$, in which case $K(2,3,r)$ is the unknot, or $r=5$, in which case $K(2,3,r)$ is the torus knot $T(3,5)$. Combining this with Corollary \ref{cor:4d-l-space} yields the following. \begin{corollary}
    For all $r\in \Z$ such that $(r,6)=1$ and $r>5$ and for all $m,n\in\Z$ such that $m+2n\equiv 2\pmod 4$, the twist-roll-spin $\tau^m\rho^n K(2,3,6s+1)$ does not bound a punctured L-space in $S^4$.
\end{corollary}

}



\bibliographystyle{alpha}
\bibliography{biblio}

@article {atiyah-bott,
    AUTHOR = {Atiyah, M. F. and Bott, R.},
     TITLE = {A {L}efschetz fixed point formula for elliptic complexes.
              {II}. {A}pplications},
   JOURNAL = {Ann. of Math. (2)},
  FJOURNAL = {Annals of Mathematics. Second Series},
    VOLUME = {88},
      YEAR = {1968},
     PAGES = {451--491},
      ISSN = {0003-486X},
   MRCLASS = {57.50},
  MRNUMBER = {232406},
MRREVIEWER = {R.\ S.\ Palais},
       DOI = {10.2307/1970721},
       URL = {https://doi.org/10.2307/1970721},
}

@article {bloom,
    AUTHOR = {Bloom, Jonathan M.},
     TITLE = {A link surgery spectral sequence in monopole {F}loer homology},
   JOURNAL = {Adv. Math.},
  FJOURNAL = {Advances in Mathematics},
    VOLUME = {226},
      YEAR = {2011},
    NUMBER = {4},
     PAGES = {3216--3281},
      ISSN = {0001-8708,1090-2082},
   MRCLASS = {57R58 (18G40 57M27)},
  MRNUMBER = {2764887},
MRREVIEWER = {Brendan\ E.\ Owens},
       DOI = {10.1016/j.aim.2010.10.014},
       URL = {https://doi.org/10.1016/j.aim.2010.10.014},
}

@article {dai,
    AUTHOR = {Dai, Irving},
     TITLE = {On the {${\rm Pin}(2)$}-equivariant monopole {F}loer homology
              of plumbed 3-manifolds},
   JOURNAL = {Michigan Math. J.},
  FJOURNAL = {Michigan Mathematical Journal},
    VOLUME = {67},
      YEAR = {2018},
    NUMBER = {2},
     PAGES = {423--447},
      ISSN = {0026-2285,1945-2365},
   MRCLASS = {57M27 (57R58)},
  MRNUMBER = {3802260},
MRREVIEWER = {Matthew\ Stoffregen},
       DOI = {10.1307/mmj/1523498585},
       URL = {https://doi.org/10.1307/mmj/1523498585},
}

@article {fox-iii,
    AUTHOR = {Fox, Ralph H.},
     TITLE = {Free differential calculus. {III}. {S}ubgroups},
   JOURNAL = {Ann. of Math. (2)},
  FJOURNAL = {Annals of Mathematics. Second Series},
    VOLUME = {64},
      YEAR = {1956},
     PAGES = {407--419},
      ISSN = {0003-486X},
   MRCLASS = {20.00 (55.00)},
  MRNUMBER = {95876},
MRREVIEWER = {R.\ C.\ Lyndon},
       DOI = {10.2307/1969592},
       URL = {https://doi.org/10.2307/1969592},
}

@article {f-lin,
    AUTHOR = {Lin, Francesco},
     TITLE = {A {M}orse-{B}ott approach to monopole {F}loer homology and the
              triangulation conjecture},
   JOURNAL = {Mem. Amer. Math. Soc.},
  FJOURNAL = {Memoirs of the American Mathematical Society},
    VOLUME = {255},
      YEAR = {2018},
    NUMBER = {1221},
     PAGES = {v+162},
      ISSN = {0065-9266,1947-6221},
      ISBN = {978-1-4704-2963-8; 978-1-4704-4819-6},
   MRCLASS = {57R58 (57M27)},
  MRNUMBER = {3827053},
MRREVIEWER = {Jianfeng\ Lin},
       DOI = {10.1090/memo/1221},
       URL = {https://doi.org/10.1090/memo/1221},
}

@article {j-lin,
    AUTHOR = {Lin, Jianfeng},
     TITLE = {The {S}eiberg-{W}itten equations on end-periodic manifolds and
              an obstruction to positive scalar curvature metrics},
   JOURNAL = {J. Topol.},
  FJOURNAL = {Journal of Topology},
    VOLUME = {12},
      YEAR = {2019},
    NUMBER = {2},
     PAGES = {328--371},
      ISSN = {1753-8416,1753-8424},
   MRCLASS = {57R57 (53C21 57K16 57R58)},
  MRNUMBER = {3911569},
MRREVIEWER = {Stefano\ Vidussi},
       DOI = {10.1112/topo.12090},
       URL = {https://doi.org/10.1112/topo.12090},
}

@article {litherland,
    AUTHOR = {Litherland, R. A.},
     TITLE = {Deforming twist-spun knots},
   JOURNAL = {Trans. Amer. Math. Soc.},
  FJOURNAL = {Transactions of the American Mathematical Society},
    VOLUME = {250},
      YEAR = {1979},
     PAGES = {311--331},
      ISSN = {0002-9947,1088-6850},
   MRCLASS = {57Q45},
  MRNUMBER = {530058},
MRREVIEWER = {Kenneth\ A.\ Perko, Jr.},
       DOI = {10.2307/1998993},
       URL = {https://doi.org/10.2307/1998993},
}

@article {konno-miyazawa-taniguchi,
    AUTHOR = {Konno, Hokuto and Miyazawa, Jin and Taniguchi, Masaki},
     TITLE = {Involutions, knots, and {F}loer {$K$}-theory},
   JOURNAL = {Compos. Math.},
  FJOURNAL = {Compositio Mathematica},
    VOLUME = {161},
      YEAR = {2025},
    NUMBER = {11},
     PAGES = {2852--2910},
      ISSN = {0010-437X,1570-5846},
   MRCLASS = {57K10 (57R58)},
  MRNUMBER = {5002153},
       DOI = {10.1112/S0010437X25102820},
       URL = {https://doi.org/10.1112/S0010437X25102820},
}

@article {lin-ruberman-saveliev,
    AUTHOR = {Lin, Jianfeng and Ruberman, Daniel and Saveliev, Nikolai},
     TITLE = {A splitting theorem for the {S}eiberg-{W}itten invariant of a
              homology {$S^1\times S^3$}},
   JOURNAL = {Geom. Topol.},
  FJOURNAL = {Geometry \& Topology},
    VOLUME = {22},
      YEAR = {2018},
    NUMBER = {5},
     PAGES = {2865--2942},
      ISSN = {1465-3060,1364-0380},
   MRCLASS = {57R57 (53C21 57M27 57R58 58J28)},
  MRNUMBER = {3811774},
MRREVIEWER = {Tirasan\ Khandhawit},
       DOI = {10.2140/gt.2018.22.2865},
       URL = {https://doi.org/10.2140/gt.2018.22.2865},
}

@article {mrowka-ozsvath-yu,
    AUTHOR = {Mrowka, Tomasz and Ozsv\'ath, Peter and Yu, Baozhen},
     TITLE = {Seiberg-{W}itten monopoles on {S}eifert fibered spaces},
   JOURNAL = {Comm. Anal. Geom.},
  FJOURNAL = {Communications in Analysis and Geometry},
    VOLUME = {5},
      YEAR = {1997},
    NUMBER = {4},
     PAGES = {685--791},
      ISSN = {1019-8385,1944-9992},
   MRCLASS = {58D27 (53C07 57R57 58E15 58G05)},
  MRNUMBER = {1611061},
MRREVIEWER = {Liviu\ I.\ Nicolaescu},
       DOI = {10.4310/CAG.1997.v5.n4.a3},
       URL = {https://doi.org/10.4310/CAG.1997.v5.n4.a3},
}

@article {n-nakamura,
    AUTHOR = {Nakamura, Nobuhiro},
     TITLE = {{$\mathrm{Pin}^-(2)$}-monopole invariants},
   JOURNAL = {J. Differential Geom.},
  FJOURNAL = {Journal of Differential Geometry},
    VOLUME = {101},
      YEAR = {2015},
    NUMBER = {3},
     PAGES = {507--549},
      ISSN = {0022-040X,1945-743X},
   MRCLASS = {57N13 (57R57)},
  MRNUMBER = {3415770},
MRREVIEWER = {Raphael\ Zentner},
       URL = {http://projecteuclid.org/euclid.jdg/1445518922},
}

@book {carter-saito,
    AUTHOR = {Carter, Scott and Kamada, Seiichi and Saito, Masahico},
     TITLE = {Surfaces in 4-space},
    SERIES = {Encyclopaedia of Mathematical Sciences},
    VOLUME = {142},
      NOTE = {Low-Dimensional Topology, III},
 PUBLISHER = {Springer-Verlag, Berlin},
      YEAR = {2004},
     PAGES = {xiv+213},
      ISBN = {3-540-21040-7},
   MRCLASS = {57Q45 (57M25 57R40)},
  MRNUMBER = {2060067},
MRREVIEWER = {Sergej\ V.\ Matveev},
       DOI = {10.1007/978-3-662-10162-9},
       URL = {https://doi.org/10.1007/978-3-662-10162-9},
}

@book {kronheimer-mrowka,
    AUTHOR = {Kronheimer, Peter and Mrowka, Tomasz},
     TITLE = {Monopoles and three-manifolds},
    SERIES = {New Mathematical Monographs},
    VOLUME = {10},
 PUBLISHER = {Cambridge University Press, Cambridge},
      YEAR = {2007},
     PAGES = {xii+796},
      ISBN = {978-0-521-88022-0},
   MRCLASS = {57R57 (53C27 57N10 57R58)},
  MRNUMBER = {2388043},
MRREVIEWER = {Vicente\ Mu\~noz},
       DOI = {10.1017/CBO9780511543111},
       URL = {https://doi.org/10.1017/CBO9780511543111},
}

@incollection {austin-braam,
    AUTHOR = {Austin, D. M. and Braam, P. J.},
     TITLE = {Morse-{B}ott theory and equivariant cohomology},
 BOOKTITLE = {The {F}loer memorial volume},
    SERIES = {Progr. Math.},
    VOLUME = {133},
     PAGES = {123--183},
 PUBLISHER = {Birkh\"auser, Basel},
      YEAR = {1995},
      ISBN = {3-7643-5044-X},
   MRCLASS = {57R70 (55N35 55N91 58E05)},
  MRNUMBER = {1362827},
MRREVIEWER = {Dave\ Auckly},
}

@unpublished{lipyanskiy,
      title = "Geometric Homology", 
      author = "Max Lipyanskiy",
      year = "2014",
      note = "preprint, \href{https://arxiv.org/abs/1409.1121}{arxiv:409.1121}"
}

@unpublished{miyazawa,
    author = "Miyazawa, Jin",
    title = "A gauge theoretic invariant of embedded surfaces in $4$-manifolds and exotic ${P}^2$-knots",
    year = "2023",
    note = "preprint, \href{https://arxiv.org/abs/2312.02041}{arXiv:2312.02041}"
}

@unpublished{miyazawa-park-taniguchi,
    author={Jin Miyazawa and JungHwan Park and Masaki Taniguchi},
    title="A satellite formula for real Seiberg-Witten Floer homotopy types", 
    year = "2025",
    note = "preprint, \href{https://arxiv.org/abs/2504.03270}{arXiv:2504.03270}"
}

@unpublished{kang-park-taniguchi,
    author = "Kang, Sungkyung and Park, JungHwan and Taniguchi, Masaki",
    title = "Cables of the figure-eight knot via real {F}r{\o}yshov invariants",
    year = "2024",
    note = "preprint, \href{https://arxiv.org/abs/2405.09295}{arxiv:2405.09295}"
}

\end{document}